\RequirePackage{fix-cm}
\documentclass{svjour3}                     
\smartqed  
 \makeatletter \let\cl@chapter\relax \makeatother
\usepackage{amsmath}
\usepackage{comment}
\usepackage[caption=false]{subfig}
\usepackage{cases}
\usepackage{array}
\usepackage{amssymb}
\usepackage{graphicx}
\usepackage{epstopdf}
\usepackage{algorithm}
\usepackage{algpseudocode}
\usepackage{algorithmicx}
\usepackage{hyperref}
\usepackage{cleveref}
\usepackage{xcolor}
\usepackage[left=3.5cm,right=3.5cm]{geometry}

\crefname{equation}{}{}

\newtheorem{assumption}{}

\crefname{assumption}{}{}

\renewcommand{\thefigure}{\thesection.\arabic{figure}}

\makeatletter
\@namedef{ver@amsmath.sty}{}
\makeatother
\usepackage{amstext}
 \journalname{Numerische Mathematik}

\newcommand{\bfs}[1]{{\boldsymbol #1}}

\begin{document}

\title{Continuous Data Assimilation for Displacement in a Porous Medium\thanks{This work was funded by NSF grant DMS 1418838.}
}


\author{H. Bessaih \and V. Ginting \and B. McCaskill}

\institute{H. Bessaih \at
              Mathematics \& Statistics, University of Wyoming, Laramie, WY, USA \\
              \email{bessaih@uwyo.edu}
           \and
           V. Ginting \at
                         Mathematics \& Statistics, University of Wyoming, Laramie, WY, USA \\
              \email{vginting@uwyo.edu}
                         \and
           B. McCaskill \at
                         Mathematics \& Statistics, University of Wyoming, Laramie, WY, USA \\
              \email{bmccaski@uwyo.edu}
}


\maketitle

\begin{abstract}
In this paper we propose the use of a continuous data assimilation algorithm for miscible flow models in a porous medium. In the absence of initial conditions for the model, observed sparse measurements are used to generate an approximation to the true solution. Under certain assumption of the sparse measurements and their incorporation into the algorithm it can be shown that the resulting approximate solution converges to the true solution at an exponential rate as time progresses. Various numerical examples are considered in order to validate the suitability of the algorithm. 
\keywords{Data Assimilation \and Miscible Flow \and Porous Media \and Nudging \and Downscaling}
\end{abstract}

\section{Introduction}
Computational and mathematical modeling is strongly motivated by the desire to predict future states of dynamical systems. However, the use of mathematical models to describe a system of interest becomes impractical if the initial state of the system can not be accurately described. For example, the construction of an accurate model for the spread of a contaminant as it infiltrates a porous subsurface is unrealistic if its concentration profile can not be measured at a sufficiently fine resolution. In contrast, it is reasonable to assume that a series of observational measurements at sparse spatial locations and times can be obtained. By properly handling this data, the current state of the dynamical system can be predicted.

The idea of feedback control can arguably be traced back to the notion of Luenberger observer (see for example \cite{MR3363684,MR1205006,1101300}). Utilization of this feedback control in combination with a Bayesian probabilistic method to weather prediction is seen in \cite{lorenc86}. Applications to data assimilation for problems modeled by Navier-Stokes equations were investigated, for example, in
\cite{MR3008183,MR3078113,2013:Titi}. The data assimilation is performed by
 introducing a set of sparse measurements as a feedback control term in the governing model. The resulting model problem then generates an approximate solution that tends toward the reference solution. Furthermore, under certain assumptions of regularity of the data it can be shown that this convergence will occur at an exponential rate. This approach provides a very practical and efficient way to find reasonable estimates of the current and future states of a dynamical system when sparse set of data measurements are available. A numerical validation study of this method for the 2D Navier-Stokes equations with periodic boundary conditions has previously been performed in \cite{Gesho}. In that study, a finite set of Fourier modes of the true solution are used to determine the feedback control term. Numerical experiments in this reference demonstrate that an exponential rate of convergence is achieved only when the value of the relaxation parameter for the feedback control term is chosen within an appropriate interval. This feature is consistent with theoretical results about the existence upper and lower bounds for this relaxation parameter.

In this work we incorporate the data assimilation algorithm introduced by Azouani, Olson, and Titi \cite{2013:Titi} with a model for miscible flow and transport to predict the spread of a contaminated fluid through a porous domain in the absence of information about initial condition. The underlying assumption is that sparse measurements of the contaminant concentration are available. In practice such measurements can be obtained in a relatively inexpensive fashion using direct current resistivity or frequency and time domain electromagnetic methods \cite{2010:Bear}. The proposed procedure gives a viable alternative to the reconstruction of the initial state of an aquifer from these measurements. Rather, the measurements can immediately be inserted into existing numerical schemes for computing the spread of the contaminant.  Furthermore, since the initial condition for the data assimilation algorithm can be chosen arbitrarily, it permits the use of smooth initial conditions in its simulation.  

The flow and transport model with which the data assimilation is integrated is in the form of a coupled system of elliptic-parabolic equations that governs the pressure and concentration of the contaminated fluid. The investigation combines theoretical and computational approaches to study this system. Although similar models have been considered (see for example \cite{2015:Farhat}), it is to the best of the authors knowledge that an application of a data assimilation technique to this particular model has not been previously made. The manner in which the system in coupled is highly nonlinear and presents many unique challenges in its analysis. In this paper, we prove the existence of a weak solution of the data assimilation algorithm and carry out several estimates in appropriate functional spaces.  When the relative permeability $\kappa$ is assumed to be Lipschitz continuous, we obtain better estimates on the pressure term, see Lemma
\ref{thm:wconv}. Moreover, under this assumption, good bounds on the data assimilation approximation 
$\hat{\theta}$  are obtained and the uniqueness of weak solutions for the data assimilation algorithm as well. We then use these estimates to establish the convergence in time of the assimilated solution to the true solution.

A numerical validation study the convergence is conducted using a first order-time marching scheme. The scheme utilizes a nodal polynomial interpolant at the sparse spatial scale to define the feedback control term. Theoretical results gathered from the analysis suggest that the convergence of the data assimilation solution depends on both the length of the sparse spatial scale and the chosen value of the relaxation parameter. Through numerical experiments, we examine  the effects of these parameters on the convergence rates of data assimilation algorithm. The applicability of the proposed methodology is further demonstrated in a prediction of a salt-water intrusion into a fresh water aquifer.

This paper is organized as follows. In \Cref{sec:preliminaries} we present the model problem of interest, define the notation and functional settings, and discuss the existence and properties of weak solutions to the model problem. In \Cref{sec:Data_Assimilation_Algorithm} we describe and analyze the proposed data assimilation algorithm for the model problem of interest. In \Cref{sec:existenceproof} existence of weak solutions to the data assimilation model are provided. In \Cref{sec:Error_Estimates} estimates for the convergence behavior of the assimilated solution towards the true model solution are established. In \Cref{Sec3.3}, we establish  further bounds on the data assimilation approximation $\hat{\theta}$ and  the uniqueness of the data assimilation approximation. In \Cref{sec:Numerical_Implementation} we develop a numerical scheme  to approximate solutions to the data assimilation algorithm. A numerical validation study of the proposed methodology is performed in \Cref{sec:Examples}. In \Cref{sec:example1,sec:example2} the dependence of the convergence rates of the data assimilation algorithm on the relaxation parameter of the proposed method are explored by simulating synthetic model problems with known solutions. In \Cref{sec:example3,sec:example4} the effect of the sparse data approximation scale on the resulting convergence rates is analyzed. Required $L^{\infty}$ estimates for the true concentration are proven in \Cref{sec:appendix1} using a maximum principle argument.

\section{Preliminaries}\label{sec:preliminaries}
The foundation for the present investigation is laid out in this section. \Cref{sec:modelproblem} discusses the model problem of interest followed by a description of functional spaces and notational conventions in \Cref{sec:funsetting}. The notion of weak solutions to the model problem is established in \Cref{sec:weaksolution}.

\subsection{Model Problem}\label{sec:modelproblem} 
The model problem under consideration is that of the miscible displacement of one incompressible fluid by another in a porous medium. The model which we adopt, developed by Peaceman and Rachford \cite{1972:bear,1962:Peaceman}, is described by the coupled elliptic-parabolic system

\begin{equation}\label{eq:modelproblem}
\begin{cases}
\displaystyle \phi \partial_t \theta - \nabla \cdot \left( D \nabla \theta - \boldsymbol{v}\theta \right)  + q_{\text{out}} \theta   = q_{\text{in}} \tilde{\theta}& \text{in } \Omega \times [0,T],\\
\displaystyle - \nabla \cdot \left( \kappa(\theta) \nabla p \right) = q_{\text{in}} - q_{\text{out}} & \text{in } \Omega \times [0,T],\\
\displaystyle \theta(\boldsymbol{x}, 0) = \theta_0(\boldsymbol{x}) & \text{in } \Omega,
\end{cases}
\end{equation}
where $\Omega$ that represents the aquifer is a bounded open domain in $\mathbb{R}^2$ with a Lipschitz boundary, $\partial \Omega$. The interval $[0,T] \subset \mathbb{R}$ is a time span of interest.  The differential system \cref{eq:modelproblem} is closed by imposing a set of boundary conditions on the boundary of $\Omega$, which is denoted by $\partial \Omega$.  This model dictates the behavior of two variables: the concentration of contaminated fluid, $\theta : \Omega \times [0,T] \to \mathbb{R}$ and the total fluid pressure, $p : \Omega \times [0,T] \to \mathbb{R}$. Relevant parameters for this system include $\phi$ the porosity of the medium, $D$ the diffusion coefficient, and $\kappa$ the relative permeability of the medium. Note that the system \eqref{eq:modelproblem} is coupled through the Darcy velocity of the fluid $\boldsymbol{v} = -\kappa(\theta) \nabla p$. The sums of the source terms for the contaminant and production well (sink) terms are given by the positive terms $q_{\text{in}} $ and $q_{\text{out}}$, respectively. The concentration of contaminated fluid entering the medium at the source terms is denoted by $\tilde{\theta}$. The concentration profile $\theta_0$ is the associated initial condition for this model problem.

Description of the proposed methodology is applied to the more general system
\begin{equation}\label{eq:genmodelproblem}
\begin{cases}
\displaystyle \partial_t \theta - \nabla \cdot \left( D \nabla \theta + \theta \kappa(\theta) \nabla p \right)  + q \theta = f \text{ in } \Omega \times [0,T],\\
\displaystyle - \nabla \cdot \left( \kappa(\theta) \nabla p \right) = g \text{ in } \Omega \times [0,T],\\
\displaystyle  \theta(\boldsymbol{x}, 0) = \theta_0(\boldsymbol{x}) \text{ in } \Omega,\\
\nabla \theta \cdot \boldsymbol{n} = 0, ~ \nabla p \cdot \boldsymbol{n} = 0,   \text{ on } \Gamma_{\text{N}} \times [0,T], \\
\theta = 0, ~ p = 0,  \text{ on } \Gamma_{\text{D}} \times [0,T],
\end{cases}
\end{equation}
where $g : \Omega \to \mathbb{R}$ and $f: \Omega \times [0,T] \to \mathbb{R}$. Specifically, this system is exactly \cref{eq:modelproblem} when $q = q_{\text{out}}, f = q_{\text{in}} \tilde{\theta}$, $g = q_{\text{in}}-q_{\text{out}}$, and the spatially dependent porosity function $\phi$ is ignored. The notation $\boldsymbol{n}$ is the unit normal vector pointing outward from $\partial \Omega$ and $\partial \Omega = \Gamma_{\text{N}} \cup \Gamma_{\text{D}}$, $\Gamma_{\text{N}} \cap \Gamma_{\text{D}} = \emptyset$ and $\Gamma_{\text{D}} \not= \emptyset$. While the proposed data assimilation algorithm is applicable for more realistic boundary conditions, the choice in the above system is to simplify presentation of the mathematical analysis. 

To give a brief overview (see \Cref{sec:Data_Assimilation_Algorithm} for a detailed exposition), the data assimilation for the above problem relies on solving
\begin{equation}\label{eq:das}
\begin{cases}
\displaystyle \partial_t \hat{\theta} - \nabla \cdot \left( D \nabla \hat{\theta} + \hat{\theta} \kappa(\hat{\theta}) \nabla \hat{p} \right)  + q \hat{\theta} + \mu P_{\hbar} (\hat{\theta}) = f +\mu P_{\hbar}(\theta) \text{ in } \Omega \times [0,T],\\
\displaystyle - \nabla \cdot \left( \kappa(\hat{\theta}) \nabla \hat{p} \right) = g \text{ in } \Omega \times [0,T],\\
\text{ same boundary conditions for $\hat{\theta}$ and $\hat{p}$},\\
\text{ a predetermined initial condition for $\hat{\theta}$.}
\end{cases}
\end{equation}
As elaborated earlier, the missing information is the initial condition of the concentration (denoted by $\theta_0$ in \cref{eq:genmodelproblem}), so as prescribed in \cref{eq:das}, the data assimilation imposes a pretty arbitrary initial condition for $\hat{\theta}$. To compensate for this missing information, the above procedure utilizes a  feedback that comes in the form of sparse measurement in space and time of the true concentration, which is represented by $P_\hbar(\theta)$. The operator $P_\hbar$ models the spatially sparse measurement of the concentration, which is assumed to be available at a scale of $\hbar$ in the porous medium $\Omega$.  The relaxation parameter, $\mu>0$, serves as a built-in tool to orient $\hat{\theta}$ to agree with the measured $\theta$ at the spatial locations and time level. Intuitively, as more available measurement information is fed to the above model, which is materialized by having increasingly smaller $\hbar$, it is expected that the produced $\hat{\theta}$ is increasingly similar to $\theta$. Achieving this behavior is dependent on the choice of $P_\hbar$ that exhibits a desirable approximation property (see \cref{eq:sparsebound}). As expounded later, the interplay between $\hbar$ and $\mu$ and a certain regularity condition of $\hat{p}$ will be crucial in establishing the convergence of $\hat{\theta}$ to $\theta$ as time progresses (see \Cref{thm:totconv}).

\subsection{Functional Settings and Notations} \label{sec:funsetting}
We employ the standard notations for the Sobolev space $W_p^k(\Omega)$, where $k \in \{0,1,\cdots\}$ and $p \in [1,\infty]$, whose norm is denoted by $\| \cdot \|_{k,p}$. Here $W^0_p(\Omega)$ is understood as the usual Lebesgue spaces over $\Omega$,  $L^p(\Omega)$.  In addition, we set $H^k(\Omega) = W^k_2(\Omega)$ and the corresponding norm to be $\| \cdot \|_k$. In the case of $k=0$, i.e., for $L^2(\Omega)$, we set $\| \cdot \|_0 = \| \cdot \|$.  The seminorms of all these spaces are defined similarly. We express the $L^2$ inner product of $u,v \in L^2(\Omega)$ with $\langle u, v \rangle := \int_{\Omega} u v ~\mathrm{d} \boldsymbol{x}$.
For the most part, we use the letter $C$ to denote constants which depend only on the domain $\Omega$. In cases where a precise bound is required, we drop this convention and specifically label the constant. For example, many of the estimates presented in this paper are established by using an interpolation inequality of Gagliardo-Nirenberg (see \cite{nirenberg59} and p. 313 of \cite{brezis2010functional}):
given $\Omega \subset \mathbb{R}^2$ and $u \in H^1(\Omega)$, then
\begin{equation} \label{eq:GNI}
\| u \|_{0,a} \le C_{\text{gn}} \| u \|^\zeta_{0,c} \, \| u \|^{1-\zeta}_1, ~~1\le c \le a < \infty, ~~\zeta = \frac{c}{a},
\end{equation}
where $C_{\text{gn}}$ depends on $\Omega$.
On the special case that $a=4$ and $c=2$, then we recover from it the Ladyzhenskaya's inequality \cite{2013:McCormick}. Since $\Omega \subset \mathbb{R}^2$, we will use a Sobolev embedding $H^1(\Omega) \hookrightarrow L^\sigma(\Omega)$ for $\sigma>2$ and $\| u \|_{0,\sigma} \le C_\text{emb} \| u \|_1$ (see for example p. 85 of \cite{200379}).
 Moreover, if a constant depends on some parameter that is not directly related to the domain $\Omega$ we emphasize this dependency with the functional notation $C(\cdot)$.

We denote by $H^1_\text{D} := \left\{ w \in H^1(\Omega) : w = 0 \text{ on } \Gamma_{\text{D}} \right\}$ with the usual norm for $H^1(\Omega)$. Note that the $H^1$ norm and semi-norm are equivalent for elements of $H^1_\text{D}$: there exists a constant $b_0 > 0$ such that $b_0 \|u\|_1 \le \|\nabla u\| \le \|u\|_1$ for $u \in H^1_{\text{D}}$. Let $\{e_k\}_{k=1}^{\infty} \subseteq  H^1_{\text{D}}$ be an orthonormal basis of $L^2(\Omega)$ that is also orthogonal in $H^1(\Omega)$. The construction of such a basis can be attained from normalizing eigenpairs of the Laplace differential operator over $\Omega$. Throughout this paper we use this basis to construct Galerkin approximations of relevant functions. To this end, for a fixed integer $m$, we denote by $\mathcal{W}^m \subset H^1_{\text{D}}$ the $\text{span} \{ e_k \}_{k=0}^{m}$, and define $\Pi_m$ to be the orthogonal projection of $H^1_{\text{D}}$ onto $\mathcal{W}^m$.

For a Sobolev space $X$ over $\Omega$ we define $L^p_T(X)$ to be the space of measurable functions $\boldsymbol{u} : [0,T] \to X$ such that for $0\le p < \infty$,
\begin{equation}
\|\boldsymbol{u}\|_{L^p_T(X)} = \left( \int_0^T \|\boldsymbol{u}(t)\|_X^p ~\mathrm{d}t \right)^{\frac{1}{p}} < \infty \text{ and }
\|\boldsymbol{u}\|_{L^\infty_T(X)} = \underset{0 \le t \le T}{\text{ess\,sup}} \|\boldsymbol{u}(t)\|_X  < \infty. 
\end{equation}
Furthermore, for a Banach space $Y$ we define $C_T^{\gamma}(Y)$ to be the space of H\"older continuous functions $\boldsymbol{u} : [0,T] \to Y$ equipped with norm
\begin{equation}
\|\boldsymbol{u}\|_{C_T^{\gamma}(Y)} = \sup_{\substack{s \in [0,T]}} \| \boldsymbol{u}(s) \|_Y + \sup_{\substack{s,t \in [0,T]\\ {s \not=t}}}\dfrac{\|\boldsymbol{u}(s) - \boldsymbol{u}(t)\|_Y}{(s-t)^{\gamma}}.
\end{equation}

\subsection{Weak Solution} \label{sec:weaksolution}
A description of weak solution of \cref{eq:modelproblem} is as follows.
\begin{definition}\label{def:weaksolution}
The pair $(\theta, p) \in L^{\infty}_T (L^2(\Omega)) \cap L_T^{2} (H^1_{\text{D}}) \times L^{\infty}_T (H^1_{\text{D}})$ is a weak solution of the model problem \cref{eq:genmodelproblem} if $\partial_t \theta \in L_T^{2}\left(H^{-1}(\Omega)\right)$, $0 \le \theta(\boldsymbol{x},t) \le 1$ for almost every $(\boldsymbol{x},t) \in \Omega \times [0,T]$, and 
\begin{subnumcases}{} \label{eq:genvarform}
\displaystyle \langle  \partial_t \theta, \psi \rangle + A(\theta, p,\psi; \theta)  =  \langle f, \psi \rangle ~\forall \psi \in H^1_{\text{D}}, \label{eq:subnuma}  \\
\displaystyle B(p,\varphi; \theta) = \langle g, \varphi \rangle ~\forall \varphi \in H^1_{\text{D}}, \label{eq:subnumb}   \\
\displaystyle \langle {\theta} (\cdot, 0), \psi \rangle = \langle {\theta}_0, \psi \rangle, ~\forall \psi \in H^1_{\text{D}}, \label{eq:subnumc}
\end{subnumcases}
 where
\begin{subequations}\label{eq:genvarforms}
\begin{align}
A(v,u,w;z) &= \langle D \nabla v, \nabla w \rangle + \langle v\kappa(z) \nabla u , \nabla w \rangle + \langle q v , w \rangle,  \label{eq:genvarformsa} \\
 B(u,w;z) &= \left\langle \kappa(z) \nabla u, \nabla w \right\rangle \label{eq:genvarformsb}
\end{align}
\end{subequations}
\end{definition}

The following assumptions on the input data are imposed to guarantee the existence of a solution according to \Cref{def:weaksolution}: 
\begin{itemize}
\item[] \begin{assumption}\label{as:init}  $\theta_0 \in L^{2}(\Omega)$ with $0 \le \theta_0(\boldsymbol{x}) \le 1$ for almost every $\boldsymbol{x} \in \Omega$, \end{assumption}
\item[] \begin{assumption}\label{as:dif} $D \in L^{\infty}(\Omega)$ with $0 < D_* < D(\boldsymbol{x}) \le D^*$  for almost every $\boldsymbol{x} \in \Omega$, \end{assumption}
\item[] \begin{assumption}\label{as:kappa}  $\kappa \in C^{0}(\mathbb{R}) \cap L^{\infty}(\mathbb{R})$ with $0 < \kappa_* \le \kappa(\zeta) \le \kappa^*$ for almost every $\zeta \in \mathbb{R}$, \end{assumption}
\item[] \begin{assumption}\label{as:q}  $q \in L^2(\Omega)$, \end{assumption} 
\item[] \begin{assumption}\label{as:g}  $g \in L^2(\Omega)$, \end{assumption}
\item[] \begin{assumption}\label{as:f}  $f \in L^{2}_T( L^2(\Omega))$, \end{assumption}
\item[] \begin{assumption}\label{eq:infbdassumptions} $g(\boldsymbol{x}) + 2q(\boldsymbol{x}) \ge 0$ and $g(\boldsymbol{x}) + q(\boldsymbol{x}) \ge f(\boldsymbol{x},t) \ge 0$ for almost every $(\boldsymbol{x},t) \in \Omega \times [0,T]$. \end{assumption}
\end{itemize}

\begin{theorem}\label{thm:modelexistence}
Under assumptions \cref{as:kappa,as:dif,as:g,as:f,as:q,as:init,eq:infbdassumptions} there exists a weak solution $(\theta, p)$ in the sense of \Cref{def:weaksolution}. Furthermore 
\begin{align}
& \|p\|_{L_T^{\infty}\left(H^1_{\text{D}}\right)} \le (\kappa_* b_0^2)^{-1}\|g\|\label{eq:truepbound}\\
&\| \theta \|_{L_T^{\infty}\left(L^2(\Omega)\right)}^2 \le  \left( \|\theta_0 \|^2 + \|f\|_{L_T^2(L^2(\Omega))}^2\right) e^{T} \label{eq:truethetainfbound}\\
&\|\theta\|_{L^{2}_T (H^1_{\text{D}})}^2 \le (D_* b_0^2)^{-1}  \left(\|\theta_0\|^2 + \|f\|_{L^{2}_T\left(L^2(\Omega)\right)}^2\right) e^{T}.\label{eq:truethetal2bound}
\end{align}
\end{theorem}

 In the interest of brevity, formal proof of \Cref{thm:modelexistence} is omitted, with a note that it can be achieved in a similar manner as analysis presented in \Cref{sec:existenceproof}.
The model problem \cref{eq:modelproblem} is well studied with extensive literature devoted to establishing the existence of solutions that satisfy \Cref{def:weaksolution}. Let us emphasize an important feature of the model that makes the analysis completely different. When $\kappa=\kappa(\boldsymbol{x})$, that is a function of $\boldsymbol{x}$ only, then the model is only one way coupled.  Model with this setting has been studied by Droniou and Talbot \cite{2014:Talbot}. On the other hand, when $\kappa=\kappa(\theta)$, then the two PDEs are strongly coupled. Strong solutions for such a system have been found to exist for the stationary model (see for example \cite{1991:Mikelic}).  Chen and Ewing have performed existence studies for both single and two phase flow models in petroleum reservoirs \cite{1999:Chen}. In addition, Fabrie and Gallou\"{e}t have found similar results with more general assumptions on $f$ and $g$ \cite{2000:Fabrie}. It is worth noting that  under the assumptions that have been utilized in the present investigation, uniqueness of the weak solution to the model remains open. Required regularity estimates are difficult to obtain due to the persistent nonlinear coupling.  By using the maximum principle, it was shown (see for example \cite{1999:Chen}) that  $0 \le \theta \le 1$ almost everywhere in $\Omega \times [0,T]$ for the general functions $f$ and $g$.

\section{A Data Assimilation Algorithm} \label{sec:Data_Assimilation_Algorithm}
\setcounter{equation}{0}
\setcounter{figure}{0}
This section devises a methodology that can be used to approximate $\theta$ when $\theta_0(\boldsymbol{x})$ is entirely unknown. The algorithm crucially relies on a set of collected sparse spatial and temporal measurements of $\theta$, which is incorporated to the governing model problem through a control term. To this end let $\mu$ be a positive relaxation parameter and $P_{\hbar}:H^1(\Omega) \to L^2(\Omega)$ be a linear operator which interpolates its input function at a length scale $\hbar > 0$. The approximation $(\hat{\theta}, \hat{p}):\Omega \times [0,T] \to (\mathbb{R},\mathbb{R})$ is set to satisfy the following formulation.

\begin{definition}\label{def:daweaksolution}{}
Given an arbitrary $\hat{\theta}_0 \in L^2(\Omega)$, find $(\hat{\theta}, \hat{p}) \in L_T^{\infty}\left(L^2(\Omega)\right) \cap L_T^{2}\left(H^1_{\text{D}}\right) \times L_T^{\infty} (H^1_{\text{D}})$ that is governed by
\begin{subnumcases}{}\label{eq:davarform}
\displaystyle \langle  \partial_t \hat{\theta}, \psi \rangle + A(\hat{\theta}, \hat{p}, \psi; \hat{\theta})  =  \langle f - \mu P_{\hbar}(\hat{\theta}-\theta) , \psi \rangle ~\forall \psi \in H^1_{\text{D}},  & \label{eq:davarform1}\\
\displaystyle B(\hat{p},\varphi;\hat{\theta}) = \langle g, \varphi \rangle ~\forall \varphi \in H^1_{\text{D}}, & \label{eq:davarform2}\\
\displaystyle \langle \hat{\theta} (\cdot, 0), \psi \rangle = \langle \hat{\theta}_0, \psi \rangle ~\forall \psi \in H^1_{\text{D}} \label{eq:davarform3},
\end{subnumcases}
and for almost every $t \in [0,T]$, with $\partial_t \hat{\theta} \in L_T^{2}\left(H^{-1} (\Omega)\right)$.
\end{definition}

Convergence of $\hat{\theta}$ towards $\theta$ as $t\to \infty$ in the appropriate metric hinges on having some constraints on the quality of the interpolation and the length scale of $\hbar$. In the present investigation, we employ a nodal-based piecewise polynomial of first degree that is continuous over $\Omega$. It has been established that under some relatively flexible assumptions (see for example \cite{2008:Brenner}) that 
\begin{equation}\label{eq:sparsebound}
\|P_{\hbar}(u) - u\| \le c_0 \hbar^k \|u \|_k, ~\forall u \in H^k(\Omega), ~k=1,2.
\end{equation}
It is also known that piecewise polynomial interpolation of this type preserves the sign
of the function that it interpolates, namely,
\begin{equation} \label{eq:Phsign}
P_\hbar (u) \ge 0 \text{ in } \Omega \text{ if } u \ge 0 \text{ in } \Omega.
\end{equation}
Furthermore, for the purpose of analysis of the data assimilation procedure, we require that
\begin{equation}\label{eq:muhbarbound}
\mu c_0^2 \hbar^2 < D_*.
\end{equation}
Other assumptions on $P_{\hbar}$ have been used (see e.g., \cite{2015:Farhat}), but approximation properties similar to \cref{eq:sparsebound} and \cref{eq:Phsign} are readily met by a large class of interpolation operators, such as splines (see for example \cite{DeBoor:1428148} and B-splines (see for example \cite{hollig03}).
The well-posedness of \Cref{def:daweaksolution} is stated next, whose proof is laid out in \Cref{sec:existenceproof}.

\begin{theorem}\label{thm:existence}
Provided that assumptions \cref{as:kappa,as:kappa,as:dif,as:g,as:f,as:q,as:init,eq:muhbarbound,eq:infbdassumptions,eq:sparsebound} are satisfied there exists a $(\hat{\theta}, \hat{p})$ in the sense of \Cref{def:daweaksolution}. Moreover, 
\begin{align}
& \|\hat{p}\|_{L_T^{\infty}\left(H^1_{\text{\em D}}\right)} \le (\kappa_* b_0^2)^{-1} \|g\|\label{eq:hatpbound}\\
&\| \hat{\theta}\|_{L_T^{\infty}\left(L^2(\Omega)\right)}^2 \le \beta  e^{\mu T}\label{eq:hatthetainfbound}\\
&\|\hat{\theta}\|_{L^{2}_T (H^1_{\text{\em D}})}^2 \le b_0^{-2}(D_* - \mu c_0^2 {\hbar}^2)^{-1} \beta e^{\mu T}, \label{eq:hatthetal2bound}
\end{align}
where
\begin{equation}\label{eq:alphabetadef}
\begin{aligned}
 \beta &= \| \hat{\theta}_0\|^2 + \frac{1}{\mu} \|f\|_{L_T^2\left(L^2(\Omega)\right)}^2
  +\mu (b_0 c_0 \hbar + 1)^2 (D_* b_0^2)^{-1} \left(\|\theta_0\|^2 + \|f\|_{L^{2}_T\left(L^2(\Omega)\right)}^2 \right) e^T.
\end{aligned}
\end{equation}
\end{theorem}

\subsection{Existence of \texorpdfstring{$(\hat{p}, \hat{\theta})$}{TEXT}}\label{sec:existenceproof}

This section gives a formal proof of \Cref{thm:existence} by applying Schauder's fixed point theorem to the mapping $S: L_T^{2}\left(L^2(\Omega)\right) \to L^2_T(L^2(\Omega))$ defined via
\begin{subnumcases}{\label{eq:sepvarform}}
\displaystyle {\langle \partial_t S z, \psi \rangle +  A(Sz,u,\psi;z) = \langle f - \mu P_{\hbar}(Sz-\theta) , \psi \rangle}, \label{eq:sepparabolic}\\
\displaystyle B(u,\varphi;z) = \langle g, \varphi \rangle,  \label{eq:sepellptic}\\
\displaystyle \langle S z(\cdot,0), \psi \rangle = \langle \hat{\theta}_0, \psi \rangle
\end{subnumcases}
for every $\psi \in H^1_{\text{D}}$, $\varphi \in H^1_{\text{D}}$ and for almost every $t \in [0,T]$. In fact, we will show that the range of $S$ is characterized by a subset of $L_T^2\left(L^2(\Omega)\right)$. In particular, it will be shown that the evaluation of each form in \eqref{eq:sepparabolic} is well defined since $Sz \in L_T^{\infty}\left(L^2(\Omega)\right) \cap L_T^{2}\left(H^1_{\text{D}}\right)$. By showing that $S$ meets the criteria of Schauder's fixed point theorem, one can conclude that $S$ has a fixed point. By construction, this fixed point is the function $\hat{\theta}$ which satisfies \cref{eq:davarform}.

Observe that the map defined by \eqref{eq:sepvarform} depends on the profile of $u$ through its associated Darcy velocity. Existence of $u$ satisfying \eqref{eq:sepellptic} is established in the following lemma.

\begin{lemma}\label{lem:elliptic}
Given ${z \in L_T^{2}\left(L^{2}(\Omega)\right)}$ and $g \in L^2(\Omega)$, there is ${u \in L_T^\infty(H^1_{\text{\em{D}}})}$ satisfying \eqref{eq:sepellptic} and
\begin{equation}
\|u\|_1 \le (\kappa_* b_0^2)^{-1} \|g\|. 
\end{equation}
 Furthermore,
there is a positive number $r_0$ with $2<r_0\le \infty$ such that $u \in L_T^\infty(W^1_r(\Omega))$ for $r \in [2,r_0)$, and
\begin{equation} \label{eq:lala}
\| u \|_{1,r} \le C(r) \| g \|,
\end{equation}
where $r_0$ only depends on $\kappa_*$, $\kappa^*$, $\Omega \cup \Gamma_{\text{N}}$ and $C(r)$ depends on  $\kappa_*$, $\kappa^*$, $\Omega$, and $r$.
\end{lemma}

\begin{proof}
Let $t\in [0,T]$ be arbitrary and the problem is to find $u = u(\cdot,t) \in H^1_{\text{D}}$ that is governed by \eqref{eq:sepellptic}.
Indeed, $B(\cdot, \cdot; z)$ is bounded and coercive in $H^1_{\text{D}}$ by \cref{as:kappa} and the Poincar\'{e} inequality, and \cref{as:g} guarantees that $\langle g, \cdot \rangle$ is bounded in $H^1_{\text{D}}$. Since $t$ was chosen arbitrarily, this variational formulation is satisfied almost everywhere in $[0,T]$, and thus the existence and uniqueness of this $u$ is guaranteed by the Lax-Milgram theorem \cite{2008:Brenner}.  Furthermore, replacing $\varphi$ in \eqref{eq:sepellptic} by $u \in H^1_{\text{D}}$, using \cref{as:kappa} and applying the Cauchy-Schwarz inequality result in $\kappa_* \| \nabla u\|^2 \le \langle g, \varphi \rangle =  \|g\| \, \|u\|$.
Noting that the $H^1$ seminorm and $H^1$ norm are equivalent over $H^1_{\text{D}}$, and $\|u\| \le \|u\|_1$ by definition, it is concluded that
 \begin{equation}\label{eq:genpbound}
\|u\|_1 \le (\kappa_* b_0^2)^{-1} \|g\|. 
 \end{equation}

Furthermore,  by Meyers' type estimate (see for example \cite{MR159110,MR990595,MR1789483}), there is a positive number $r_0$ with $r_0>0$ such that if $u \in H^1_\text{D}$ is governed by
$B(u,\varphi ; z) = F(\varphi)$ for every $\varphi \in H^1_\text{D}$, and $F \in W^{-1}_r(\Omega)$ for $r \in [2,r_0)$, then $u \in W^1_r(\Omega)$. In particular, there is a $C(r)$ that depends only on $\kappa_*$, $\kappa^*$, $\Omega$, and $r$ such that
\begin{equation} \label{eq:gugu}
\| u \|_{1,r} \le \tilde{C}(r) \| F \|_{-1,r}.
\end{equation}
The constant $r_0$ only depends on $\kappa_*$, $\kappa^*$, $\Omega \cup \Gamma_{\text{N}}$. In our situation, $F(\varphi) = \langle g, \varphi \rangle$ where $g \in L^2(\Omega)$. By Cauchy-Schwarz and H\"older's inequalities
\begin{equation}
|F(\varphi)| \le \| g \| \, \| \varphi \| \le \| g \| \, |\Omega|^{\frac{r-2}{2r}} \| \varphi \|_{0,r},
\end{equation}
which implies that $\| F \|_{-1,r} \le  \| g \| \, |\Omega|^{\frac{r-2}{2r}}$. Using this in \cref{eq:gugu} gives \cref{eq:lala}. The whole proof is complete.
\qed
\end{proof}

The above bounds are true for any $z\in L_T^{2}\left(L^{2}(\Omega)\right)$ and holds for almost every $t \in [0,T]$. Consequently, the bounds for $p$ and $\hat{p}$ presented in \cref{eq:truepbound}  can be viewed as a consequence of \cref{eq:genpbound}. Also, bounds for $p$ and $\hat{p}$ in the fashion of \cref{eq:lala} will be used in the forthcoming analysis. We remark that utilization of Meyers' type estimate in the analysis miscible displacement in a porous medium has been done in \cite{MR1185634}.

It turns out the Meyers' type estimate in the preceding lemma in combination with Lipschitz continuity of $\kappa$ can be used to establish a quantification of discrepancy of the pressure in terms of discrepancy in the coupling caused by $\kappa$. The following lemma states this result.

\begin{lemma}\label{thm:wconv}
Assume that $\kappa$ is Lipschitz continuous, i.e.,
\begin{equation} \label{eq:LipsK}
|\kappa(\zeta_1) - \kappa(\zeta_2)| \le L_\kappa |\zeta_1 - \zeta_2|, ~~\zeta_1,\zeta_2 \in \mathbb{R}.
\end{equation}
Given $z_j \in H^1_\text{D}$, $j=1,2$, let $u_j \in H^1_\text{D}$ be governed by
\begin{equation}
B(u_j,\varphi ; z_j) = \langle g, \varphi \rangle, ~~ \varphi \in H^1_\text{D}.
\end{equation}
 Then
\begin{equation} \label{eq:erru}
\| \nabla (u_1 - u_2) \| \le C_{\boldsymbol{w}} C(r) \, \| g \| \, \| z_1 - z_2 \|^{2/s}  \, \| \nabla (z_1 - z_2) \|^{(s-2)/s},
\end{equation}
and
\begin{equation} \label{eq:errw}
\| \kappa(z_1)\nabla u_1 - \kappa(z_2) \nabla u_2 \| \le C_{\boldsymbol{w}} C(r) \, \| g \| \, \| z_1 - z_2 \|^{2/s}  \, \| \nabla (z_1 - z_2) \|^{(s-2)/s},
\end{equation}
where $\frac{1}{r} + \frac{1}{s} = \frac{1}{2}$, and $C_{\boldsymbol{w}}$ depends on $\kappa_*$, $\kappa^*$, $L_\kappa$, $b_0$, $C_{\text{\em{gn}}}$, and $C(r)$ is as in \cref{eq:lala} of \Cref{lem:elliptic}.
\end{lemma}
\begin{proof}
To simplify the presentation, denote
$$
\varepsilon_{\boldsymbol{v}} = \kappa(z_1)\nabla u_1 - \kappa(z_2) \nabla u_2,~~ \varepsilon_u = u_1 - u_2, ~~\varepsilon_z = z_1 - z_2.
$$
By H\"{o}lder equality and applying the Lipschitz continuity of $\kappa$,
\begin{equation} \label{eq:rrr}
\| (\kappa(z_1) - \kappa(z_2) )\nabla u_1 \| \le \| \kappa(z_1) - \kappa(z_2) \|_{0,s} \, \| \nabla u_1 \|_{0,r} \le L_\kappa \| \varepsilon_z \|_{0,s} \, \| \nabla u_1 \|_{0,r},
\end{equation}
where $\frac{1}{r} + \frac{1}{s} = \frac{1}{2}$.
Furthermore, by Gagliardo-Nirenberg inequality \cref{eq:GNI} and Poincare inequality (since each $z_j \in H^1_{\text{D}})$,
\begin{equation} \label{eq:etagn}
\| \varepsilon_z \|_{0,s} \le C_{\text{gn}} \| \varepsilon_z \|^{2/s}  \, \| \varepsilon_z \|_1^{(s-2)/s} \le C_{\text{gn}} \,b_0^{(2-s)/s} \, \| \varepsilon_z \|^{2/s}  \, \| \nabla \varepsilon_z \|^{(s-2)/s}. 
\end{equation}
Applying \cref{eq:etagn} and \cref{eq:lala} to \cref{eq:rrr} gives
\begin{equation} \label{eq:sss}
\| (\kappa(z_1) - \kappa(z_2) )\nabla u_1 \| \le L_\kappa C(r) C_{\text{gn}} \,b_0^{(2-s)/s} \| g \|  \, \| \varepsilon_z \|^{2/s}  \, \| \nabla \varepsilon_z \|^{(s-2)/s}.
\end{equation}

Since $B(u_1,\varphi ; z_1) = B(u_2,\varphi ; z_2)$ for every $\varphi \in H^1_{\text{D}}$, 
$$
 \begin{aligned}
 B(\varepsilon_u,\varphi ; z_2) &=  B(u_1,\varphi ; z_2) - B(u_2,\varphi ; z_2) \\
 &= B(u_1,\varphi ; z_2) - B(u_1,\varphi ; z_1)\\
 & = \langle (\kappa(z_2) - \kappa(z_1) ) \nabla u_1, \nabla \varphi \rangle.
  \end{aligned}
 $$
 Therefore, by choosing $\varphi = \varepsilon_u$ in the above identity and using \cref{as:kappa}, the following estimate can be performed: 
 $$
 \begin{aligned}
\kappa_* \| \nabla \varepsilon_u \|^2 &\le \langle (\kappa(z_2) - \kappa(z_1)) \nabla u_1, \nabla \varepsilon_u \rangle
\le \|(\kappa(z_2) - \kappa(z_1)) \nabla u_1\| \, \|\nabla \varepsilon_u \|,
\end{aligned}
 $$
and thus $ \| \nabla \varepsilon_u \| \le \kappa_*^{-1}   \|(\kappa(z_2) - \kappa(z_1)) \nabla u_1\|$ so that its combination with \cref{eq:sss} yields
\begin{equation}
\| \nabla \varepsilon_u \| \le \kappa_*^{-1} \, L_\kappa \, C_{\text{gn}} \,b_0^{(2-s)/s} \, C(r)  \, \| g \|  \, \| \varepsilon_z \|^{2/s}  \, \| \nabla \varepsilon_z \|^{(s-2)/s}.
\end{equation}

Finally, by adding and subtracting $\kappa(z_2)\nabla u_1$ in $\varepsilon_{\boldsymbol{v}}$, using the triangle inequality and \cref{eq:sss},
\begin{equation}\label{eq:wbound}
\begin{aligned}
\|\varepsilon_{\boldsymbol{v}}\| &\le \| (\kappa(z_1) - \kappa(z_2) )\nabla u_1 \| + \|\kappa(z_2) \nabla \varepsilon_u\| \\
&\le \| (\kappa(z_1) - \kappa(z_2) )\nabla u_1 \| + \kappa^* \| \nabla \varepsilon_u\|\\
&\le (1 + \kappa^*  \kappa_*^{-1} ) \| (\kappa(z_1) - \kappa(z_2) )\nabla u_1 \|\\
&\le (1 + \kappa^*  \kappa_*^{-1} ) L_\kappa  C_{\text{gn}} \,b_0^{(2-s)/s} \, C(r) \, \| g \|  \, \| \varepsilon_z \|^{2/s}  \, \| \nabla \varepsilon_z \|^{(s-2)/s}.
\end{aligned}
\end{equation}
The proof is completed by setting
$$
C_{\boldsymbol{w}} = \max\{ \kappa_*^{-1},  1 + \kappa^*  \kappa_*^{-1} \} \, L_\kappa  C_{\text{gn}} \,b_0^{(2-s)/s}.
$$
\qed
\end{proof}

The next investigation is on the construction of a sequence of Galerkin approximations of $S z$ in finite-dimensional subspace of $H^1_{\text{D}}$. Set $\hat{\theta}_0^m = \sum_{k=0}^m \langle \hat{\theta}_0, e_k \rangle e_k$ and note that it is an element of the space $\mathcal{W}^m$ defined in \Cref{sec:preliminaries}. For $z \in L_T^2\left(L^2(\Omega)\right)$ let $S_m: L_T^{2}\left(L^2(\Omega)\right) \to L_T^{2}\left(L^2(\Omega)\right)$ be defined via
\begin{equation}\label{eq:smdef}
\begin{cases}
&\langle \partial_t S_m z, \psi \rangle +  A(S_mz,u,\psi;z)~\mathrm{d}s =   \langle f - \mu P_{\hbar}(S_m z-\theta) , \psi \rangle\\
&\langle S_m z(\cdot,0), \psi \rangle = \langle \hat{\theta}_0^m, \psi \rangle,
\end{cases}
\end{equation}
for all $\psi \in \mathcal{W}^m$ and for almost every $t \in [0,T]$, where $u$ is the solution of \eqref{eq:sepellptic}. Recall from \Cref{sec:funsetting} that $\mathcal{W}^m \subset H^1_{\text{D}}$. The existence of $S_m z$ is established in the following lemma. 

\begin{lemma}\label{lem:galerkinapprox}
Given $z \in L_T^2\left(L^2(\Omega)\right)$ there is ${S_m z} :\Omega \times [0,T] \to \mathbb{R}$ satisfying \cref{eq:smdef}. 
\end{lemma}
\begin{proof}
Let $\alpha_{k}: [0,T] \to \mathbb{R}$, $k=1 \ldots m$, and set
$(S_m z)(\boldsymbol{x}, t) = \sum_{k=0}^m \alpha_{k}(t) e_k(\boldsymbol{x})$.
The intention is to characterize $\alpha_k$ in such a way that $S_m z$ satisfies \cref{eq:smdef} over $\mathcal{W}^m$. Using $e_i$ as test functions in \cref{eq:smdef} and applying the orthogonality of this basis yields a system of ODEs for $\boldsymbol{\alpha} = [ \alpha_{0}, \dots \alpha_{m} ]^{\top}$:
\begin{equation}\label{eq:integraleq}
\boldsymbol{\alpha}^\prime(t) + M \boldsymbol{\alpha} = \boldsymbol{f}, ~~
\boldsymbol{\alpha}(0) =  \boldsymbol{\alpha}_0,
\end{equation}
where $\boldsymbol{\alpha}_0 = [ \langle\hat{\theta}_0, e_0 \rangle, \cdots, \langle\hat{\theta}_0, e_m \rangle]^\top$, and the entries of $M$ and $\boldsymbol{f}$ are given by
\begin{equation}
M_{ik} = A(e_k,u,e_i;z) + \left\langle  \mu P_{\hbar}(e_k), e_i \right\rangle
\text{ and }  f_i = \langle f + \mu P_{\hbar}(\theta), e_i \rangle, ~ i,k=1,\cdots,m.
\end{equation}
In accordance with standard existence theory for first order IVPs (see e.g. \cite{2017:Meiss}), there exists a unique continuous ${\boldsymbol{\alpha} : [0,T] \to \mathbb{R}^{m+1}}$ governed by \cref{eq:integraleq}.
\qed
\end{proof}

\begin{lemma}\label{lem:range}
$S_m z$ in \eqref{eq:smdef} belongs to $L_T^{\infty}\left(L^2(\Omega)\right) \cap L_T^2\left(H^1_{\text{D}}\right)$ for $z \in L_T^2\left(L^2(\Omega)\right)$.
\end{lemma}

\begin{proof}
Using $\psi = S_m z$ in \eqref{eq:smdef} along with \cref{as:q,as:dif} gives
\begin{equation}\label{eq:finvarformctest}
\begin{aligned}
 \frac{1}{2} \frac{\mathrm{d}}{\mathrm{d}{t}} \|  S_m z \|^2 \, + & \, \frac{D_*}{2}  \| \nabla S_m z \|^2  + J_4 \le J_1 + J_2 + J_3,
\end{aligned}
\end{equation}
with
\begin{equation}
\begin{aligned}
J_1 &= \langle f, S_m z \rangle,\\
J_2 &= - \mu\langle  P_{\hbar}(S_m z), S_m z \rangle,\\
J_3 &= \mu \langle  P_{\hbar}(\theta), S_m z \rangle,\\
J_4 &= \langle S_m z \, \kappa(z) \nabla u, \nabla S_m z \rangle + \langle q S_m z, S_mz \rangle \\
\end{aligned}
\end{equation}
Here the task is to bound each of these terms.

Cauchy-Schwarz and Young's inequalities give
\begin{equation}\label{eq:fmubound}
J_1 \le \frac{1}{2 \mu} \|f\|^2 + \frac{\mu}{2} \|S_m z\|^2.
\end{equation} 

Adding and subtracting $\mu \langle S_m z, S_m z \rangle$ and using \cref{eq:sparsebound} gives
\begin{equation}\label{eq:finctrbound}
\begin{aligned}
J_2 &= \mu \langle S_m z -  P_{\hbar}(S_m z), S_m z\rangle - \mu \| S_m z\|^2\\
&\le \mu \|P_{{\hbar}}(S_m z) - S_m z\|  \|S_m z\| - \mu \|S_m z\|^2\\
&\le \frac{\mu c_0^2 {\hbar}^2}{2} \| \nabla S_m z\|^2 - \frac{\mu}{2}\|S_m z\|^2.
\end{aligned}
\end{equation}

Applying Cauchy-Schwarz and Young's inequalities gives
\begin{equation}\label{eq:ctrltermbound}
J_3 \le \frac{\mu}{2} \|P_{\hbar}(\theta)\|^2 + \frac{\mu}{2} \|S_m z\|^2,
\end{equation}
where in lieu of the stability of $P_\hbar$:
\begin{equation}
 \|P_{\hbar}(\theta) \| \le \| P_{\hbar}(\theta) - \theta\| + \|\theta \| \le c_0 \hbar \|\nabla \theta\| + \|\theta\| \le (b_0 c_0 \hbar +1) \|\theta\|_1,
\end{equation}
yield
\begin{equation}\label{eq:ctrtermbd}
J_3 \le \frac{\mu}{2} (b_0 c_0 \hbar + 1)^2 \|\theta\|_1^2 + \frac{\mu}{2} \|S_m z\|^2. 
\end{equation}

Since $\nabla (S_m z)^2 = 2 S_m z \nabla S_m z$ and $u$ satisfies \eqref{eq:sepellptic},
\begin{equation}\label{eq:advecbound1}
\begin{aligned}
J_4 &= \frac{1}{2} \langle \kappa(z) \nabla u , \nabla (S_m z)^2 \rangle + \langle q S_m z, S_mz \rangle\\
&= \frac{1}{2} \langle g, (S_m z)^2 \rangle +  \langle q S_m z, S_mz \rangle\\
&= \frac{1}{2} \langle g + 2q, (S_m z)^2 \rangle\\
&\ge 0,
\end{aligned}
\end{equation}
where we have used \cref{eq:infbdassumptions} to get the bound.

Putting all these estimates back to \cref{eq:finvarformctest} yields an observation that $S_m z$ satisfies
\begin{equation}\label{eq:finvarformtestbounded}
\begin{aligned}
& \frac{\mathrm{d}}{\mathrm{d}{t}} \|  S_m z \|^2 + \left(D_* - \mu c_0^2 {\hbar}^2\right) \| \nabla S_m z\|^2  \le 
 \mu \| S_m z \|^2  + \frac{1}{\mu} \|f\|^2 + \mu (b_0 c_0 \hbar + 1)^2 \|\theta\|_1^2,
\end{aligned}
\end{equation}
which due to $D_* - \mu c_0^2 {\hbar}^2 > 0$ in \cref{eq:muhbarbound}, along with integration over $(0,t)$ further  yields
\begin{equation}
 \|  S_m z \|^2  \le \int_0^t  \left( \mu\| S_m z \|^2 + \frac{1}{\mu} \|f\|^2 + \mu (b_0 c_0 \hbar + 1)^2 \|\theta\|_1^2 \right) ~\mathrm{d} s  + \| \hat{\theta}_0^m\|^2.
\end{equation}
By the integral form of Gr\"{o}nwall's inequality \cite{2002:Majada},
\begin{equation}
\|S_m z\|^2 \le \left(  \| \hat{\theta}_0^m\|^2  + \int_0^t \Big(\frac{1}{\mu} \|f\|^2 + \mu (b_0 c_0 \hbar + 1)^2 \|\theta\|_1^2  \Big)~\mathrm{d} s \right)  e^{\mu t}.
\end{equation}
Taking the supremum over all $t \in [0,T]$  and noting that $\theta$ satisfies \cref{eq:truethetal2bound} it follows that
\begin{equation}\label{eq:thetaml2bound}
 \|  S_m z \|^2  \le \beta e^{\mu t},
\end{equation}
where $\beta$ is as defined in \cref{eq:alphabetadef}. It then follows that $S_m z \in L_T^{\infty}\left(L^2(\Omega)\right)$ as desired.  

Furthermore, returning to \cref{eq:finvarformtestbounded} and integrating it over $(0,T)$ and noting that $\| S_m z \|^2$ is a positive quantity one observes that
\begin{equation}
\left(D_* - \mu c_0^2 {\hbar}^2\right)  b_0^2  \|S_m z\|_{L_T^2(H^1_{\text{D}})}^2 \le \mu \int_0^T \| S_m z \|^2 ~\mathrm{d} s + \beta.
\end{equation}
Since $D_* - \mu c_0^2 \hbar^2 > 0$ and $S_m z \in L_T^{\infty}\left(L^2(\Omega)\right)$ satisfies \cref{eq:thetaml2bound} we conclude that
\begin{equation}\label{eq:varformtestbounded}
\| S_m z\|_{L_T^2(H^1_{\text{D}})}^2 \le b_0^{-2}(2D_* - \mu c_0^2 {\hbar}^2)^{-1} \beta e^{\mu T}.
\end{equation}
Consequently,  $S_m z \in L_T^2(H^1_{\text{D}})$ as desired.
\qed
\end{proof}

In \Cref{lem:holdercontinuity} we turn our attention to establishing a sense of continuity of $S_m z$ with respect to time. It is worth noting that since $S z$ and consequently $\hat{\theta}$ are obtained by passing to the limit on $S_m z$ the bounds appearing in \cref{eq:hatthetainfbound,eq:hatthetal2bound} are an immediate consequence of \cref{eq:thetaml2bound} and \cref{eq:varformtestbounded}. Similarly, it is also the case that $S z$ and $\hat{\theta}$ can be shown to be equicontinuous with respect to time using the same process for proving \Cref{lem:holdercontinuity}. 
\begin{lemma}\label{lem:holdercontinuity}
Let $z \in L_T^2(L^2(\Omega)$ be fixed and $S_m z$ be the solution of \cref{eq:smdef}. Then $\partial_t S_m z \in L_T^2\left( H^{-1}(\Omega)\right)$ and $ S_m z \in C_T^\gamma(H^{-1}(\Omega))$ for $0\le\gamma \le \frac{1}{2}$.
\end{lemma}
\begin{proof}
Since as a test function $\psi \in H^1_\text{D}$, it follows from \Cref{lem:range} that $\partial_t S_m z \in L_T^2 \left( H^{-1}(\Omega) \right)$. 
Integration of \cref{eq:smdef} over $(s,t)$ gives
\begin{equation}\label{eq:cequation}
|\langle  S_m z(\cdot,t) - S_m z(\cdot, s), \psi \rangle| \le L_1 + L_2 + L_3 + L_4,
\end{equation}
where
$$
\begin{aligned}
L_1 &=   \int_s^t |\langle D \nabla S_m z, \nabla \psi \rangle| \, \mathrm{d} \tau, ~~
L_2 =   \int_s^t |\langle S_m z \, \kappa(z) \nabla u, \nabla \psi \rangle| \, \mathrm{d} \tau\\
L_3 &=   \int_s^t |\langle q S_m z, \psi \rangle| \, \mathrm{d} \tau, ~~
\hspace*{0.7cm}  L_4 = \int_s^t |\langle f - \mu P_{\hbar}(S_m z-\theta) , \psi \rangle| \, \mathrm{d} \tau 
\end{aligned}
$$
Estimates for each of the above terms are shown next.

Using \cref{as:dif} and Cauchy-Schwarz inequality it holds that
\begin{equation}\label{eq:cdiffbound}
L_1  \le D^* \|\nabla \psi \| \, \|S_m z\|_{L_T^2(H^1_{\text{D}})} (t-s)^{\frac{1}{2}}.
\end{equation}

By Cauchy-Schwarz and H\"older's inequalities along with Sobolev embedding $H^1(\Omega) \hookrightarrow L^\sigma(\Omega)$ for $\sigma>2$ (see for example p. 85 of \cite{200379}), and \cref{eq:lala},
\begin{equation}
\begin{aligned}
|\langle S_m z \, \kappa(z) \nabla u, \nabla \psi \rangle|
&\le \kappa^* \| S_m z \nabla u \| \, \| \nabla \psi \| \\
&\le \kappa^* \| S_m z \|_{0,\sigma} \, \| \nabla u \|_{0,r} \, \| \nabla \psi \|, \text{ where } \frac{1}{\sigma}+\frac{1}{r}  = \frac{1}{2}  \\
&\le \kappa^* C_\text{emb} \| S_m z \|_1 \, C(r) \| g \| \, \| \nabla \psi \|.
\end{aligned}
\end{equation}
Using this last inequality to estimate $L_2$ gives
\begin{equation}\label{eq:cadvecbound}
\begin{aligned}
L_2 &\le \kappa^* C_\text{emb} C(r) \| g \| \, \| \nabla \psi \|  \int_s^t \| S_m z\|_1 \, {\rm d} \tau\\
&\le \kappa^* C_\text{emb} C(r) \| g \| \, \| \nabla \psi \| \, \| S_m z\|_{L_T^2(H^1_{\text{D}})} (t-s)^{\frac{1}{2}}.
\end{aligned}
\end{equation}

Using Cauchy-Schwarz and H\"older's inequalities and applying Sobolev embedding $H^1(\Omega) \hookrightarrow L^4(\Omega)$,
we find that
\begin{equation}
| \langle q S_m z, \psi \rangle | \le   \| q \| \, \| S_m z \, \psi \| \le 
\| q \| \, \| S_m z \|_{0,4} \, \| \psi \|_{0,4} \le C^2_\text{emb} \| q \| \, \| S_m z \|_1 \, \| \psi \|_1.
\end{equation}
Application of this last estimate to $L_3$ yields
\begin{equation}\label{eq:csourcebound}
L_3 \le C^2_\text{emb} \| q \| \, \| \psi \|_1 \int_s^t  \, \| S_m z \|_1  \mathrm{d} \tau \le C^2_\text{emb} \| q \| \, \| \psi \|_1 \, \| S_m z \|_{L^2_T(H^1_\text{D})} (t-s)^{\frac{1}{2}}.
\end{equation}

By Cauchy-Schwarz and triangle inequalities, and \cref{eq:sparsebound} we find that
\begin{equation}\label{eq:csparsebound}
\begin{aligned}
L_4 &\le \|\psi\|  \int_s^t \big( \|f\| + \mu \| P_{\hbar} ( S_m z - \theta) - \left(S_m z - \theta\right) \| + \mu \| S_m z - \theta \| \big) ~\mathrm{d} \tau \\
&\le \|\psi\| \int_s^t \big( \|f\| + \mu c_0 \hbar \|\nabla(S_m z -\theta)\|  + \mu \| S_m z - \theta \| \big) ~\mathrm{d} \tau\\
&\le C(\hbar, \mu) \|\psi\|  \left( \|f\|_{L_T^2 \left(L^2(\Omega)\right)} + \|S_m z - \theta\|_{L_T^2(H^1_{\text{D}})} \right)(t-s)^{\frac{1}{2}}.
\end{aligned}
\end{equation}

Combining \cref{eq:cdiffbound}, \cref{eq:cadvecbound}, \cref{eq:csourcebound}, and \cref{eq:csparsebound} with \cref{eq:cequation} and taking the supremum over all $s,t \in [0,T]$ yield
\begin{equation*}
\hspace{-0.cm}\|S_m z\|_{C_T^{\frac{1}{2}}\left(H^{-1}(\Omega)\right)}  \le C\left(\|g \| + \|q\| + \|f\|_{L_T^2 \left(L^2(\Omega)\right)} + \|S_m z - \theta\|_{L_T^2(H^1_{\text{D}})} + \|\theta\|_{L_T^2(H^1_{\text{D}})}\right),
\end{equation*}
confirming that $S_m z \in C_T^{\gamma} (H^{-1}(\Omega))$ for $0\le\gamma \le \frac{1}{2}$.
\qed 
\end{proof}

\begin{proposition}\label{thm-Sz}
For a given $\hat{\theta}_0 \in L^2(\Omega)$, $z \in L_T^{2}\left( L^2(\Omega) \right)$ and $u$ the solution of \eqref{eq:sepellptic}, there exists a unique solution $S z \in L_T^{\infty}\left(L^2(\Omega)\right) \cap L_T^2\left(H^1_{\text{D}}\right)$ of \eqref{eq:sepparabolic}.
\end{proposition}
\begin{proof}
As a consequence of \Cref{lem:range,lem:holdercontinuity} we find that the sequence of functions $\{S_m z\}_{m=0}^{\infty}$ is bounded in $L_T^2\left(H^1_{\text{D}}\right)$ and in $C_T^{\gamma}(H^{-1}(\Omega))$. Using Theorem 4.1 in Chapter 4 in \cite{1988:Vishik} there exists a subsequence $\{S_{m_k} z\}_{k=0}^{\infty} \subset \{S_m z\}_{m=0}^{\infty}$ and  $S z$ such that 
\begin{equation}
S_{m_k} z \rightarrow S z \text{ in } L_T^2(L^2(\Omega)) \text{ and }
S_{m_k} z \rightarrow S z \text{ in } C_T(H^{-1}(\Omega) ).
\end{equation}
Moreover, we also have
\begin{equation}
S_{m_k} z \rightharpoonup S z \text{ in } L_T^2(H^1_{\text{D}}).
\end{equation}

Using an arbitrary $\psi \in H^1_{\text{D}}$ as a test function in \cref{eq:smdef} and integrating over time we find that
\begin{equation}\label{eq:testedfiniteparabolicvarform}
\langle  S_{m_k} z, \psi \rangle +  \int_0^t A(S_{m_k} z, u,\psi; z)  ~\mathrm{d} s  =   \int_0^t \langle f - \mu P_{\hbar}(S_{m_k} z-\theta) , \psi \rangle  ~\mathrm{d} s + \langle  \hat{\theta}_0^m, \psi \rangle.
\end{equation} 
 Upon passing the weak limit to each term of \cref{eq:testedfiniteparabolicvarform} we find that
\begin{equation}\label{eq:varpassedlimits}
\langle S z, \psi \rangle +  \int_0^t A(S z, u,\psi; z)  ~\mathrm{d} s  =   \int_0^t \langle f - \mu P_{\hbar}(S z-\theta) , \psi \rangle  ~\mathrm{d} s + \langle \hat{\theta}_0, \psi \rangle,
\end{equation}
for all $\psi \in  H^1_{\text{D}}$.
\qed \end{proof}

In addition to showing that $S$ satisfying \eqref{eq:sepvarform} is well-posed, we must establish that it is continuous with a relatively compact range to invoke Schauder's fixed point theorem.

\begin{lemma}
$S: L^2_T(L^2(\Omega)) \to L_T^2(L^2(\Omega))$ in \eqref{eq:sepvarform} is continuous.
\end{lemma}
\begin{proof}
Let $\{z_n\} \subset L_T^2\left(L^2(\Omega)\right)$ such that 
$\| z_n - z \|_{ L_T^2(L^2(\Omega))} \to 0$ as $n \to \infty$.
For each $z_n$ let $(S z_n, u_n)$ be the associated solution of \eqref{eq:sepvarform}. Note that by definition \eqref{eq:sepellptic} we have that
\begin{equation}\label{eq:darcydiff}
\langle \kappa(z_n) \nabla u_n - \kappa(z)\nabla u, \nabla \varphi \rangle = 0
\end{equation}
for any $\varphi \in L_T^{\infty}\left(H^1_{\text{D}} \right)$. Adding and subtracting the mixed term $\kappa(z_n) \nabla u$ in \cref{eq:darcydiff} and using $\varphi = u_n - u$ as a test function it follows that
\begin{equation}\label{eq:pressequal}
 \kappa_* \|\nabla (u_n - u) \|^2  \le  \langle(\kappa(z) - \kappa(z_n) ) \nabla u, \nabla (u_n - u) \rangle.\\
\end{equation}
Integrating \cref{eq:pressequal} over $[0,T]$, and applying the Cauchy-Schwarz inequality in space yields
\begin{equation}
 \kappa_* \|\nabla (u_n - u) \|_{L_T^2(L^2(\Omega)^2)}^2  \le \int_0^T \left\|\left(\kappa(z) - \kappa(z_n) \right) \nabla u \right\| \left \| \nabla (u_n - u)\right\| ~\mathrm{d} s.
\end{equation}
Expanding $\left\|\left(\kappa(z) - \kappa(z_n) \right) \nabla u \right\| \left \| \nabla (u_n - u)\right\|$ with Young's inequality we find that
\begin{equation}\label{eq:pnbound}
 \|\nabla (u_n- u) \|_{L_T^2(L^{2}(\Omega)^2)}^2 \le C(\kappa_*)\left\| \left(\kappa(z)  - \kappa(z_n) \right)  \nabla u \right\|_{L_T^2(L^{2}(\Omega)^2)}^2.
\end{equation}
Moreover, by adding and subtracting the mixed term $\kappa(z_n) \nabla u$, using the triangle inequality, and \cref{eq:pnbound},
$$
\begin{aligned}
\|\kappa(z_n) \nabla u_n - \kappa(z) \nabla u \|_{L_T^2(L^{2}(\Omega)^2)} &= \left \| \kappa(z_n) \nabla (u_n -u) + \left(\kappa(z_n)-\kappa(z)\right)\nabla u \right\|_{L_T^2(L^{2}(\Omega)^2)}\\
& \le C(\kappa_*) \left \|\left(\kappa(z)-\kappa(z_n)\right)\nabla u \right\|_{L_T^2(L^{2}(\Omega)^2)}.
\end{aligned}
$$

Recall that $z_n \to z$ strongly in $L_T^{2}\left(L^{2}(\Omega)\right)$. Consequently, there exists a subsequence of $\{z_n\}$ for which $z_n(\boldsymbol{x}, t) \to z(\boldsymbol{x},t)$ for almost every $(\boldsymbol{x},t)$. Then using model assumption \cref{as:kappa} we find that $\kappa(z_n) \to \kappa(z)$ almost everywhere on $\Omega \times [0,T]$. Since $\kappa$ is bounded it follows from the Lebesgue dominated convergence theorem that
\begin{equation}
\lim_{n \to \infty}  \left \|\left(\kappa(z)-\kappa(z_n)\right)\nabla u \right\|_{L_T^2(L^{2}(\Omega)^2)} = 0.
\end{equation}
It then follows that $\kappa (z_n) \nabla u_n$ converges to $\kappa(z) \nabla u$ strongly in $L_T^2(L^2(\Omega)^2)$.
Using classical compactness results it follows from \Cref{lem:range} that there exists a function $Y$ such that
$$
S z_n \rightharpoonup Y  \text{ in } L_T^\infty(L^2(\Omega))\cap L_T^2(H^1_{\text{D}}) 
\text{ and }
S z_n \rightharpoonup Y  \text{ in } C_T^{\gamma}(H^{-1}(\Omega)), \text{ for } 0 \le \gamma \le \frac{1}{2}.
$$

Since $S z_n$ satisfies \eqref{eq:sepvarform}, passing to the limit in \eqref{eq:sepparabolic} gives
\begin{equation}
\langle Y, \psi \rangle + \int_0^t A(Y,u,\psi;z)~\mathrm{d}s =  \int_0^t \langle f - \mu P_{\hbar}(Y-\theta) , \psi \rangle ~\mathrm{d} s + \langle \hat{\theta}_0, \psi \rangle
\end{equation} 
for all $\psi \in H^1_{\text{D}}$. Note that convergence of the advection term in \eqref{eq:sepparabolic} follows from the fact that $\kappa (z_n) \nabla u_n$ converges to $\kappa(z) \nabla u$ strongly in $L_T^2\left(L^2(\Omega)^2\right)$.
Since solutions to the decoupled system \eqref{eq:sepvarform} are unique we conclude that $Y = Sz$ and the whole sequence $S z_n$ converges to $S z$ in $L_T^2\left(L^2(\Omega)\right)$. Consequently, we find that $S$ is a continuous mapping as desired.
\qed \end{proof}

Recall that any function $Sz$ satisfying \eqref{eq:sepvarform} is bounded in $L_T^2\left(H^1_{\text{D}}\right)$. Furthermore, we have shown that $Sz$ is also an element of $C_T^\gamma\left(H^{-1}(\Omega)\right)$ for $0 \le \gamma \le \frac{1}{2}$. As a consequence of Theorem 4.1 in \cite{1988:Vishik} we conclude that the range of $S$ is relatively compact in $L_T^2\left(L^2(\Omega)\right)$ since $H^1(\Omega) \subset L^2(\Omega) \subset H^{-1}(\Omega)$. By Schauder's fixed point theorem there exists a fixed point for $S$. Thus, there exists a weak solution $(\hat{\theta}, \hat{p})$ which satisfies \Cref{def:daweaksolution}.

\subsection{Error Estimates for Data Assimilation Algorithm}\label{sec:Error_Estimates}
The manner in which the approximation $(\hat{\theta},\hat{p})$ converges to $(\theta,p)$ of \Cref{def:weaksolution} is disseminated by analysis of discrepancy between the two pairs: $(\eta, \boldsymbol{w})$, with 
\begin{equation} \label{eq:errornich}
\eta = \theta-\hat{\theta} \text{ and } \boldsymbol{w} =  \kappa(\theta) \nabla p - \kappa(\hat{\theta}) \nabla \hat{p}.
\end{equation}
 Subtracting \cref{eq:davarform} from \cref{eq:genvarform} gives
\begin{equation}\label{eq:DifferenceProblem}
\begin{cases}
~\displaystyle \langle \partial_t \eta, \psi \rangle + \langle D \nabla \eta, \nabla \psi \rangle + \langle \theta \boldsymbol{w}, \nabla \psi\rangle 
+ \langle \eta \kappa(\hat{\theta}) \nabla \hat{p} ,\nabla \psi\rangle
+ \langle q \eta , \psi\rangle \hspace*{-0.04cm} = \hspace*{-0.04cm}- \mu \langle P_\hbar(\eta), \psi \rangle,\\
~ \langle \boldsymbol{w}, \nabla \varphi \rangle = 0,\\
\displaystyle \langle \eta (\cdot, 0),\psi \rangle = \langle \eta_0, \psi \rangle, \\
\end{cases}
\end{equation}
for all $\varphi \in H^1_{\text{D}}$ and $\psi \in H^1_{\text{D}}$,
where identity $\theta \kappa(\theta) \nabla p - \hat{\theta} \kappa(\hat{\theta}) \nabla \hat{p} = \theta \boldsymbol{w} + \eta \kappa(\hat{\theta}) \nabla \hat{p}$ and $\eta_0 =  \theta_0 - \hat{\theta}_0$ have been used.

\begin{lemma}\label{thm:convergence}
Under the assumptions \cref{as:kappa,as:dif,as:g,as:f,as:q,as:init,eq:infbdassumptions}, \cref{eq:LipsK}, \cref{eq:sparsebound}, and \cref{eq:muhbarbound}, $\eta$  in \cref{eq:DifferenceProblem} satisfies
\begin{equation}\label{eq:errorbound}
\|\eta(\cdot,t)\|^2 \le \|\eta_0\|^2 \, e^{\xi t},  \text{ with } \xi = \left(2\tilde{C}_{\boldsymbol{w}} \| g \|^s  - \mu\right) ,
\end{equation}
for every $t>0$, where $\frac{1}{r} + \frac{1}{s} = \frac{1}{2}$, and
\begin{equation}
 \tilde{C}_{\boldsymbol{w}} = \frac{1}{s} \frac{(C_{\boldsymbol{w}} C(r))^s}{(\epsilon D_*)^{s-1}},  \text{ and } \epsilon = \frac{s}{2(s-1)}.
\end{equation}
\end{lemma}

\begin{proof}
We replace $\psi$ in the first equation of \cref{eq:DifferenceProblem} by $\eta$, use 
 \cref{as:dif,as:q} (so that $D_* \|\nabla \eta\|^2 \le \langle D \nabla \eta , \nabla \eta \rangle$ to get
\begin{equation}\label{eq:diffeqtest}
\frac{1}{2} \frac{\mathrm{d}}{\mathrm{d} t} \| \eta \|^2 + D_* \|\nabla \eta\|^2 + I_3(\eta) \le 
I_1(\eta) + I_2(\eta),
\end{equation}
where
\begin{equation} \label{eq:III}
I_1(\zeta) = \langle  -\mu P_{\hbar}(\zeta), \zeta \rangle, ~~
I_2(\zeta) = \langle \theta \boldsymbol{w}, \nabla \zeta  \rangle, ~\text{and}~
I_3(\zeta) = \langle \zeta \kappa(\hat{\theta}) \nabla \hat{p}, \nabla \zeta \rangle + \langle q \zeta,\zeta \rangle.
\end{equation}

Adding and subtracting $\mu \langle \eta,\eta \rangle$, using \cref{eq:sparsebound} and 
Cauchy-Schwarz and Young's inequalities gives
\begin{equation}\label{eq:ctrbound}
\begin{aligned}
I_1(\eta) &= \mu \langle \eta - P_{\hbar}(\eta), \eta\rangle - \mu \| \eta \|^2\\
&\le \mu \|P_{\hbar}(\eta) - \eta\|  \|\eta\| - \mu \|\eta\|^2\\
&\le \frac{\mu c_0^2 \hbar^2}{2} \| \nabla \eta\|^2 - \frac{\mu}{2}\|\eta\|^2.
\end{aligned}
\end{equation}

Since $0 \le \theta \le 1$ almost everywhere in $\Omega \times [0,T]$ we conclude by using Cauchy-Schwarz and Young's inequalities that
\begin{equation}\label{eq:mixedadvecbound}
 I_2(\eta) \le |\left\langle  \theta \boldsymbol{w}, \nabla \eta \right\rangle| \le
\|\boldsymbol{w}\| \, \|\nabla \eta \|
\le \frac{1}{2 \epsilon D_*} \|\boldsymbol{w}\|^2 + \frac{\epsilon D_*}{2} \|\nabla \eta \|^2, ~~\text{ for some } \epsilon>0. 
\end{equation}
Recognizing the expression of $\boldsymbol{w}$ in \cref{eq:errornich}, we may use
\cref{eq:errw} in \Cref{thm:wconv} along with Young's inequality to perform the following estimate:
\begin{equation} \label{eq:ijkl}
\begin{aligned}
\| \boldsymbol{w} \|^2 &\le \left( C_{\boldsymbol{w}} C(r) \, \| g \| \, \| \eta \|^{2/s}  \, \| \nabla \eta \|^{(s-2)/s} \right)^2\\
&= \left( \left(C_{\boldsymbol{w}} C(r) \, \| g \| \right)^{s} \| \eta \|^{2} \right)^{2/s}  \,
\left(\| \nabla \eta \|^2 \right)^{(s-2)/s}\\
&\le \frac{2}{s} \frac{(C_{\boldsymbol{w}} C(r) \| g \|)^s}{(\epsilon D_*)^{s-2}} \,  \| \eta \|^2 + \frac{s-2}{s} (\epsilon D_* \| \nabla \eta \|)^2.
\end{aligned}
\end{equation}
Combining this last estimate with \cref{eq:mixedadvecbound} yields
\begin{equation}
 I_2(\eta) \le \tilde{C}_{\boldsymbol{w}}  \| g \|^s \,  \| \eta \|^2 + \frac{s-2}{2s} \epsilon D_* \, \| \nabla \eta \|^2, \text{ where }
 \tilde{C}_{\boldsymbol{w}} = \frac{1}{s} \frac{(C_{\boldsymbol{w}} C(r))^s}{(\epsilon D_*)^{s-1}}
\end{equation}

An application of density argument such as done in \Cref{sec:appendix1} gives
$\langle \eta \kappa(\hat{\theta}) \nabla \hat{p}, \nabla \eta \rangle = \frac{1}{2} \langle g , \eta^2 \rangle$, which further implies
\begin{equation}  \label{eq:myI33}
 I_3(\eta)  =  \frac{1}{2} \langle g , \eta^2 \rangle + \langle q \eta,\eta \rangle =
 \frac{1}{2} \langle g + 2 q, \eta^2 \rangle \ge 0,
\end{equation}
where we have used \cref{eq:infbdassumptions}.

Collecting all the estimates for $I_j(\eta)$, $j=1,2,3$ back to \cref{eq:diffeqtest} gives
\begin{equation}\label{eq:etainequalityexpanded}
\begin{aligned}
 \frac{\mathrm{d}}{\mathrm{d} t} \| \eta \|^2 + (c(\epsilon) D_* - & \mu c_0^2 {\hbar}^2 ) \| \nabla \eta\|^2 \le  \xi \|\eta\|^2.
\end{aligned}
\end{equation}
where
\begin{equation} \label{eq:ttt}
\xi = \left(2\tilde{C}_{\boldsymbol{w}} \| g \|^s  - \mu\right) \text{ and }
c(\epsilon) = 2 - \frac{2(s-1) \epsilon}{s}.
\end{equation}
Choose $\epsilon = \frac{s}{2(s-1)}$ so that $c(\epsilon) = 1$. This choice is well defined since $s-1>0$, due to the fact that $1/r+1/s=1/2$, i.e., $r,s>2$.
Using the bound \cref{eq:muhbarbound} in \cref{eq:etainequalityexpanded} yields
\begin{equation}\label{eq:etainequality}
 \frac{\mathrm{d}}{\mathrm{d} t} \| \eta \|^2_{2} \le \xi \|\eta\|_2^2, 
\end{equation}
from which the proof is completed.
\qed \end{proof}

Without further specifications on components involved in $\xi$,
the preceding lemma guarantees the boundedness of $\| \eta \| = \| \theta - \hat{\theta} \|$ in finite time as long as $\xi$ remains bounded. However, notice that $\mu$ in the expression of $\xi$ is a free parameter, so there is a certain degree of flexibility to choose $\mu$ that renders a negative value for $\xi$.

\begin{theorem} \label{thm:totconv}
Adopt hypotheses in \Cref{thm:wconv} and \Cref{thm:convergence}. Choose $\mu > 0$ such that $\xi<0$ in \Cref{thm:convergence}. Then 
\begin{equation} \label{eq:errone}
\| \theta(\cdot,t) - \hat{\theta}(\cdot,t) \| \to 0 \text{ as } t \to \infty. 
\end{equation}
Furthermore,
\begin{equation} \label{eq:errtwo}
\| [\kappa(\theta) \nabla p - \kappa(\hat{\theta}) \nabla \hat{p}](\cdot,t) \| \to 0 \text{ as } t \to \infty. 
\end{equation}
\end{theorem}
\begin{proof}
Establishing this statement is straightforward from the inequalities in  \Cref{thm:convergence}. In particular, we choose $\mu>0$ such that
\begin{equation}
\mu > 2\tilde{C}_{\boldsymbol{w}}  \| g \|^s, 
\end{equation}
which gives $\xi<0$. Inequality in \cref{eq:errorbound} of \Cref{thm:convergence} gives \cref{eq:errone}.

Next, with the same choice of $\epsilon$ to give $c(\epsilon)=1$ in \cref{eq:ttt} and integrating  \cref{eq:etainequalityexpanded} over $(0,T)$, one obtains
\begin{equation}
 \| \eta(\cdot,t) \|^2 - \| \eta_0 \|^2 +
 (D_* - \mu c_0^2 {\hbar}^2 ) \int_0^T \| \nabla \eta(\cdot,t) \|^2  \, {\rm d} t \le  
 \xi  \int_0^T \|\eta(\cdot,t)\|^2 \, {\rm d} t,
\end{equation}
from which inequality in \Cref{eq:errorbound} is used to deduce 
\begin{equation}
(D_* - \mu c_0^2 {\hbar}^2 ) \int_0^T \| \nabla \eta(\cdot,t) \|^2  \, {\rm d} t \le
 \| \eta_0 \|^2 \left( 1 + \xi  \int_0^T  e^{\xi t} \, {\rm d} t \right) =  \| \eta_0 \|^2 e^{\xi T}.
\end{equation}
This last inequality implies that $\int_0^T \| \nabla \eta(\cdot,t) \|^2  \, {\rm d} t \to 0$ as $T \to \infty$, so $\| \nabla \eta(\cdot,t) \|$ remains bounded as $t \to \infty$. This fact together with \cref{eq:errone} are used in \cref{eq:errw} of \Cref{thm:wconv} to get \cref{eq:errtwo}. This completes the proof.
\qed
\end{proof}

\subsection{Physical Relevance of the Data Assimilation Solution}\label{Sec3.3}
After establishing its convergence, the next task is to look into the physical relevance 
of $\hat{\theta}$. As seen in \Cref{sec:appendix1}, under appropriate assumptions, the true concentration, $\theta$, maintains its value between 0 and 1 in $\Omega \times [0,T]$. The next lemma and theorem confirms the same behavior for $\hat{\theta}$ as long as $\mu$ is chosen in accordance with the setting in \Cref{thm:totconv}, and a correct choice of initial condition $\hat{\theta}_0$ is used. It is then followed by a theorem that establishes the uniqueness of this $\hat{\theta}$.

\begin{lemma} \label{thm:etanonnegative}
Assume all hypotheses in \Cref{thm:totconv} and choose $\hat{\theta}_0$ such that
$\eta_0 = \theta_0 - \hat{\theta}_0 \ge 0$. Then
\begin{equation}
(\theta(\boldsymbol{x},t) - \hat{\theta}(\boldsymbol{x},t) ) \ge 0, ~~\text{for almost every}~~(\boldsymbol{x},t) \in \Omega \times [0,T].
\end{equation}
\end{lemma}
\begin{proof}
Set $\eta = \theta - \hat{\theta}$ and write $\eta = \eta^+ - \eta^-$, where $\eta^{+} = \text{max}\{0,\eta\}$ and $\eta^{-} = \text{max}\{0,-\eta\}$ and
$$
\nabla \eta^+ =
\begin{cases}
\nabla \eta  &\text{ if } \eta>0,\\
0 & \text{ if }  \eta\le 0,
\end{cases}
\text{ and }
\nabla \eta^- =
\begin{cases}
0 &\text{ if } \eta>0,\\
-\nabla \eta & \text{ if }  \eta\le 0.
\end{cases}
$$
Using the above definition, it is clear that $-\eta \eta^- = -(\eta^+ - \eta^-) \eta^- = (\eta^-)^2$. By the same token, $\nabla \eta \cdot \nabla \eta^{-} = | \nabla \eta^- |^2$, and $\eta \nabla \eta^- = \eta^- \nabla \eta^-$. Thus, we may use $\psi = -\eta^-$ in the first equation of \cref{eq:DifferenceProblem} to get
\begin{equation}
\frac{1}{2} \frac{\mathrm{d}}{\mathrm{d} t} \| \eta^- \|^2 + D_* \|\nabla \eta^-\|^2 + I_3(\eta^-) \le 
I_1(\eta^-) + I_2(\eta^-),
\end{equation}
where $I_j(\zeta)$, $j=1,2,3$ are as described in \cref{eq:III}. Much of the remaining proof are verbatim of shown in the proof of \Cref{thm:convergence}. In particular, following the exact estimates of $I_j$, $j=1,2,3$ in the proof of that lemma, we arrive at
\begin{equation}
 \frac{\mathrm{d}}{\mathrm{d} t} \| \eta^- \|^2 \le \xi \|\eta^- \|^2 \le 0,
\end{equation}
since $\xi<0$ has been chosen as stated in \Cref{thm:totconv}. Integrating over $(0,s) \subset [0,T]$ gives $\| \eta^-(\cdot,s) \|^2 \le \| \eta^-(\cdot,0) \|^2 = \| \eta^-_0 \|^2$. By the stated choice of $\hat{\theta}_0$,  $\eta_0 \ge 0$, so this means $\| \eta^-_0 \| = 0$. It follows that  $\| \eta^-(\cdot,s) \| = 0$, and thus $\eta^-(\boldsymbol{x},s) = 0$ for almost every $(\boldsymbol{x},s) \in \Omega \times [0,T]$. This implies that $\eta(\boldsymbol{x},s) \ge 0$.
\qed
\end{proof}

\begin{theorem} \label{eq:hatthetabdd}
Assume hypotheses in \Cref{thm:totconv} and choose $\hat{\theta}_0 \ge 0$ such that
$\eta_0 = \theta_0 - \hat{\theta}_0 \ge 0$. Then
\begin{equation}
\hat{\theta}(\boldsymbol{x},t) \in [0,1], ~~\text{for almost every}~~(\boldsymbol{x},t) \in \Omega \times [0,T].
\end{equation}
\end{theorem}
\begin{proof}
Since $\eta_0 \ge 0$, \Cref{thm:etanonnegative} implies $\hat{\theta}(\boldsymbol{x},t) \le \theta(\boldsymbol{x},t)$. Moreover, because \Cref{sec:appendix1} shows that $\theta(\boldsymbol{x},t) \le 1$, it follows that $\hat{\theta}(\boldsymbol{x},t) \le 1$.  

Next we show that $\hat{\theta}(\boldsymbol{x},t) \ge 0$. To proceed, we repeat the first part of the proof in \Cref{sec:appendix1} verbatim. In this case we write $\hat{\theta} = \hat{\theta}^+ - \hat{\theta}^-$, where $\hat{\theta}^{+} = \text{max}\{0,\hat{\theta} \}$ and $\hat{\theta}^{-} = \text{max}\{0,-\hat{\theta} \}$ use $\psi = -\hat{\theta}^-$ in \eqref{eq:davarform1} to get
\begin{equation}\label{eq:hinfbdtest}
 \frac{1}{2} \frac{\mathrm{d} }{\mathrm{d} t} \| \hat{\theta}^- \|^2 +A(\hat{\theta}^-, \hat{p},\hat{\theta}^-; \theta) = -\langle f- \mu P_{\hbar}(\hat{\theta}-\theta), \hat{\theta}^- \rangle,
\end{equation} 
and after redoing much of the steps in \Cref{sec:appendix1}, we arrive at
\begin{equation}\label{eq:bbq}
 \frac{1}{2} \frac{\mathrm{d} }{\mathrm{d} t} \| \hat{\theta}^{-} \|^2  + \langle f + q {\hat{\theta}^{-} }+ \frac{1}{2} g \hat{\theta}^{-}, \hat{\theta}^{-} \rangle + \langle \mu P_{\hbar}(\theta-\hat{\theta}), \hat{\theta}^- \rangle\le 0.
\end{equation}
As in \Cref{sec:appendix1}, assumption \cref{eq:infbdassumptions} implies $\langle f + q \theta^{-} +\frac{1}{2} g \theta^{-} , \theta^{-} \rangle \ge 0$. Furthermore, because \Cref{thm:etanonnegative} guarantees that $(\theta - \hat{\theta})\ge 0$, it follows by \cref{eq:Phsign} that   $P_\hbar(\theta - \hat{\theta}) \ge 0$, and thus $\langle \mu P_{\hbar}(\theta-\hat{\theta}), \hat{\theta}^- \rangle \ge0$.  Integration of \cref{eq:bbq} over $(0,s) \subseteq [0,T]$ and application of these two inequalites give
$\| \hat{\theta}^{-}(\cdot, s) \|_{2}^2 - \| \hat{\theta}^{-}(\cdot,0) \|_{2}^2 \le 0$.
By assumption $\hat{\theta}_0(\boldsymbol{x}) \ge 0$ for almost every $\boldsymbol{x} \in \Omega$ so it follows that $\hat{\theta}^-(\boldsymbol{x},s) = 0$ for almost every $(\boldsymbol{x},s) \in \Omega \times [0,T]$. Consequently, $\hat{\theta}(\boldsymbol{x},t) \ge 0$ for almost every $(\boldsymbol{x},t) \in \Omega \times [0,T]$. This completes the proof.
\qed
\end{proof}

At this stage, we are in a position to establish the uniqueness of $(\hat{\theta},\hat{p})$. 

\begin{theorem}
Under the assumptions in \Cref{thm:existence}  and \Cref{eq:hatthetabdd}, the solution $(\hat{\theta}, \hat{p})$ of the data assimilation model \cref{eq:davarform} is unique in the class  $L_T^{\infty}\left(L^2(\Omega)\right) \cap L_T^{2}\left(H^1_D\right) \times L_T^{\infty} (H^1_{D})$.
\end{theorem}

\begin{proof}
To the contrary, assume each $(\hat{\theta}_j, \hat{p}_j)$, $j=1, 2$ is the solution to \cref{eq:davarform}. Designating $\hat{\eta} =\hat{\theta}_1-\hat{\theta}_2$, and $\hat{\boldsymbol{w}} = \kappa( \hat{\theta}_1)  \nabla \hat{p}_1- \kappa( \hat{\theta}_2)  \nabla \hat{p}_2$, they are governed by
\begin{equation}\label{eq:Differencehat}
\hspace*{-0.15cm}\begin{cases}
~\displaystyle \langle \partial_t \hat{\eta}, \psi \rangle + \langle D \nabla \hat{\eta}, \nabla \psi \rangle + \langle \hat{\theta}_1 \hat{\boldsymbol{w}}, \nabla \psi\rangle 
+ \langle \hat{\eta} \kappa(\hat{\theta}_2) \nabla \hat{p}_2 ,\nabla \psi\rangle
+ \langle q \hat{\eta} , \psi\rangle \hspace*{-0.04cm} = \hspace*{-0.04cm}- \mu \langle P_\hbar(\hat{\eta}), \psi \rangle,\\
~ \langle \hat{\boldsymbol{w}}, \nabla \varphi \rangle = 0,\\
\displaystyle \langle \hat{\eta} (\cdot, 0),\psi \rangle = 0,
\end{cases}
\end{equation}
for every $\varphi \in H^1_{\text{D}}$ and $\psi \in H^1_{\text{D}}$. Notice the structural resemblance of \cref{eq:Differencehat} to \cref{eq:DifferenceProblem}, so the same estimation procedures done \Cref{thm:convergence} can be applied here. With the help of \Cref{eq:hatthetabdd} and those procedures, we arrive at
\begin{equation*}
\|\hat{\eta}(\cdot,t)\|^2 \le \|\hat{\eta}(\cdot,0)\|^2 \, e^{\xi t} = 0,
\end{equation*}
which infers that for almost every $(\boldsymbol{x},t) \in \Omega \times (0,T]$, $\hat{\eta}(\boldsymbol{x}, t)=0$, and thus $\hat{\theta}_1 = \hat{\theta}_2$ almost everywhere in $\Omega \times (0,T]$. Also, using \cref{eq:erru} in \Cref{thm:wconv},
\begin{align*}
\|\nabla(\hat{p}_1 - \hat{p}_2 \| \leq C_{\boldsymbol{w}} C(r) \, \| g \| \, \| \hat{\theta}_1 - \hat{\theta}_2 \|^{2/s}  \, \| \nabla (\hat{\theta}_1 - \hat{\theta}_2) \|^{(s-2)/s},
\end{align*}
which implies that $\|\nabla(\hat{p}_1 - \hat{p}_2)\|=0$. Since $p_j(\cdot,t)  \in H^1_\text{D}$, this means that $\hat{p}_1 = \hat{p}_2$ almost everywhere in  $\Omega \times (0,T]$. The proof is complete.
\qed
\end{proof}

\section{A Numerical Implementation of the Data Assimilation}\label{sec:Numerical_Implementation}
\setcounter{equation}{0}
\setcounter{figure}{0}

\subsection{Setting of the Computational Domain and the Assimilation Data}
Description of the numerical implementation is focused on performing the data assimilation methodology on a rectangular domain $\Omega = (0,L_1) \times (0,L_2)$ such as depicted on \cref{fig:disco}. In this $\Omega$, it is assumed that there is a set $\{ {\bfs x}_i \}_{i=0}^{N_{\text{sp}}} \subset \Omega$ of spatial measurement locations with a maximum distance $\hbar$, over which certain functional values of $\theta$ are known. In particular, we assume that for every ${\bfs x}_i$ there exists a bounded linear functional $\gamma_i : H^1_{\rm D}(\Omega) \to \mathbb{R}$. The set of these functional is the ground for constructing the interpolation operator $P_{\hbar}: H^1_{\rm D}(\Omega) \to \text{span} \{ \phi_{\hbar,i}: i=0,\cdots, N_{\text{sp}} \}$:
\begin{equation}
	P_{\hbar}(z) = \sum_{i=0}^{N_{\text{sp}}} \gamma_i(z) \phi_{\hbar,i},
\end{equation}
where $\phi_{\hbar,i} : \overline{\Omega} \to [0,1]$ is the usual nodal bilinear basis function over the partition of $\Omega$ caused by $\{ {\bfs x}_i \}_{i=0}^{N_{\text{sp}}}$, such that $\phi_{\hbar,i}(\boldsymbol{x}_j) = \delta_{ij}$. This kind of interpolation operator has been well studied and can be shown to satisfy \cref{eq:sparsebound} under sufficient regularity of the measured function and the measurement locations \cite{2008:Brenner}. An example of such a $\gamma_i$ is the average value of the function at ${\bfs x}_i$:
$$
\gamma_i(z) = \frac{1}{\omega_i} \int_{\omega_i} z(\boldsymbol{x}) \, {\rm d} {\bfs x},
$$
where $\omega_i \subset \Omega$ is a certain region associated with ${\bfs x}_i$.
\begin{figure}\label{fig:disc}
	\vspace*{-0.5cm}
	\subfloat[]{\label{fig:disco} \includegraphics[width=0.32\columnwidth]{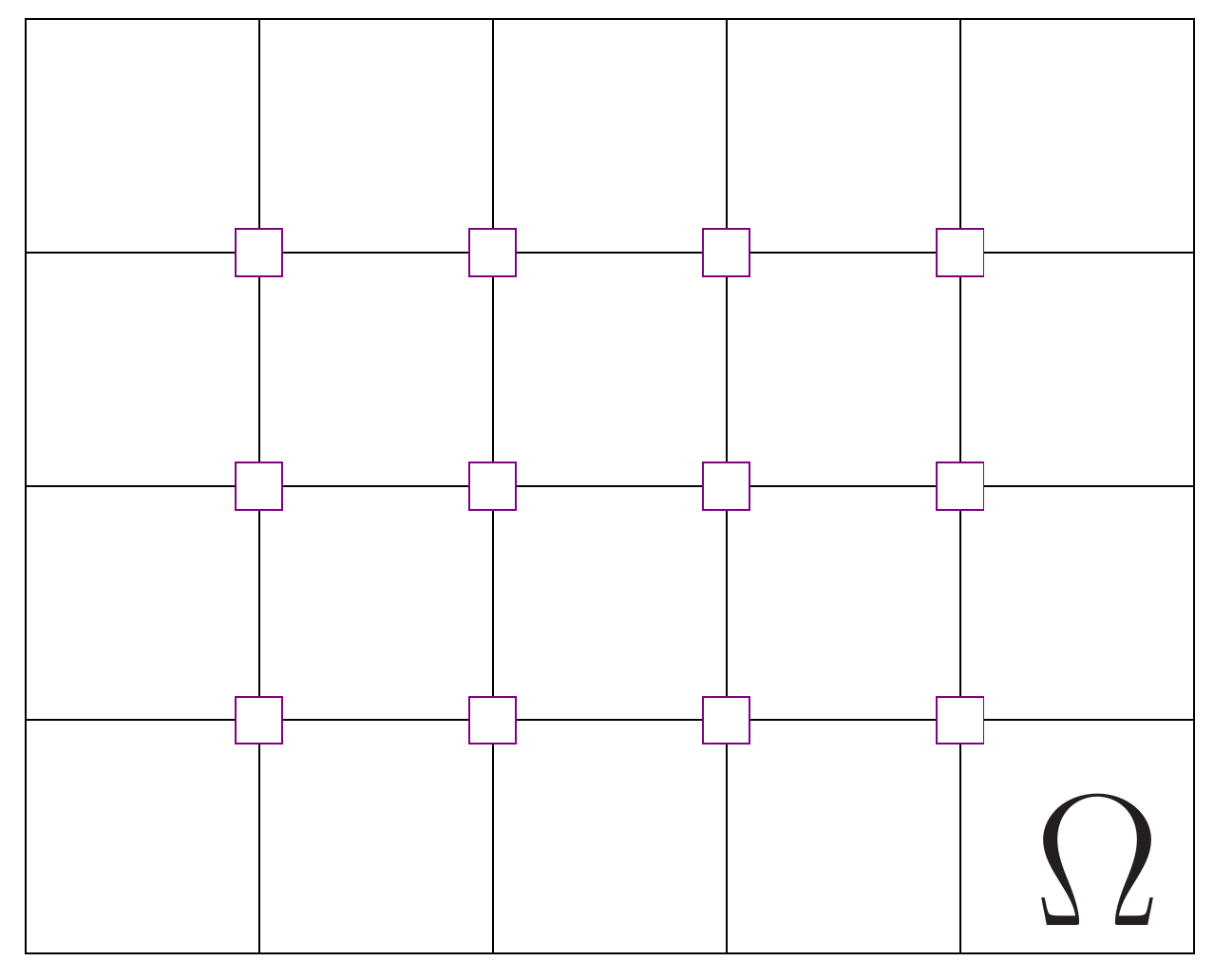}}
	\subfloat[]{\label{fig:disct} \includegraphics[width=0.32\columnwidth]{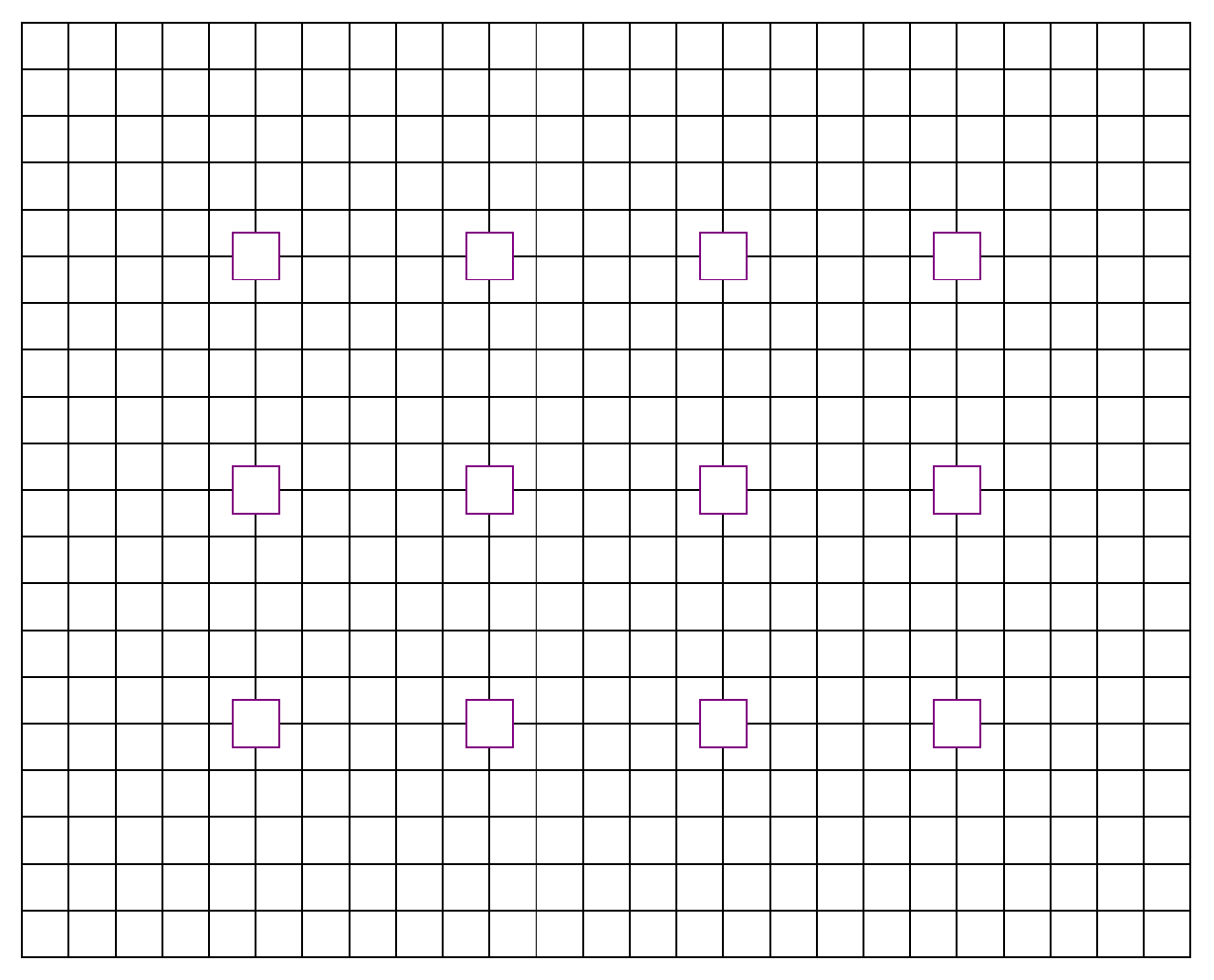}}
	\subfloat[]{\label{fig:discd} \includegraphics[width=0.32\columnwidth]{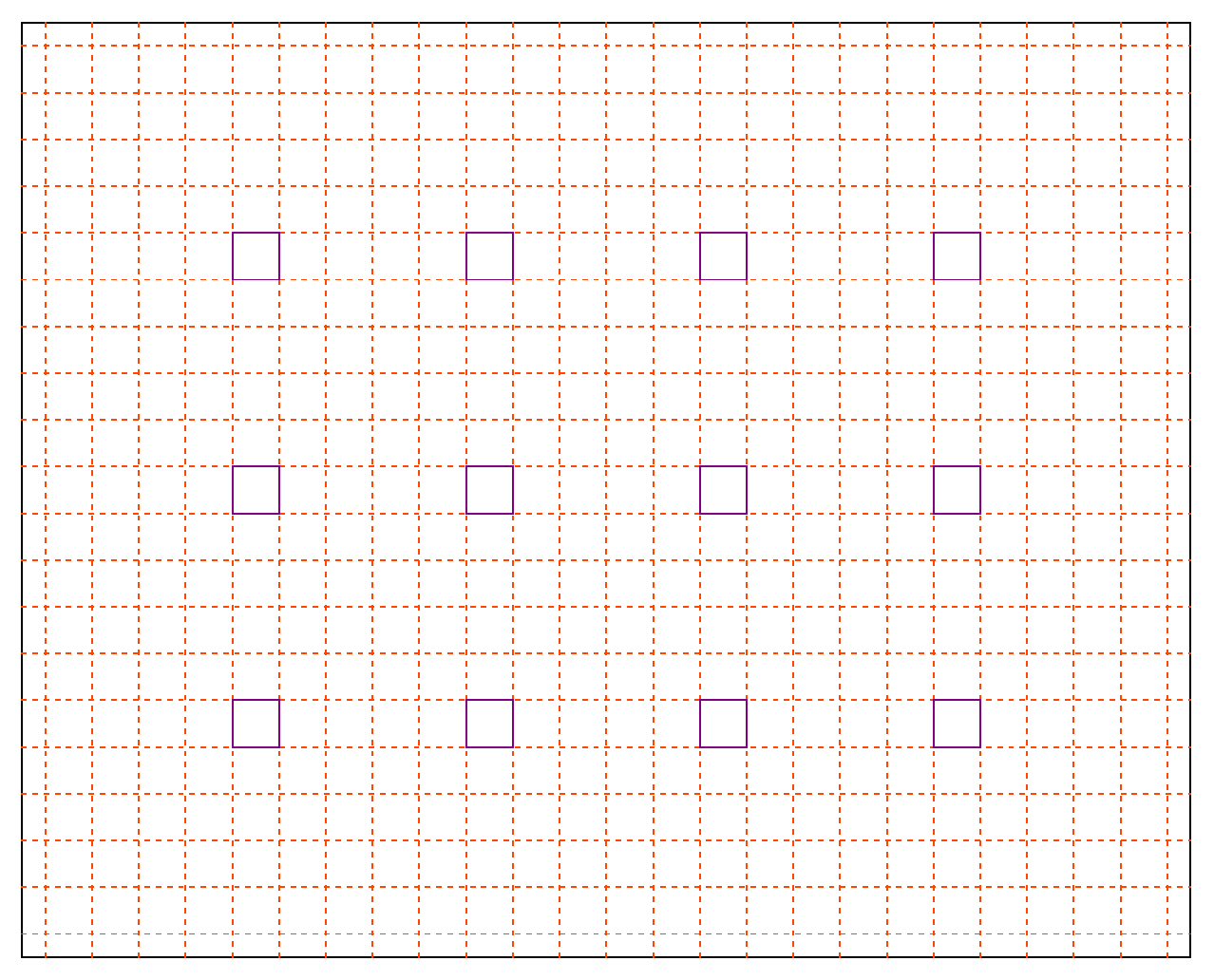}}
	\caption{An example rectangular domain $\Omega$ with sparse data measurement locations (left), the associated triangulation of the domain into rectangular element $\mathcal{T}_h$ (middle), and a corresponding dual mesh of control volumes $\mathcal{D}_h$ (right).}
\end{figure}

\subsection{Numerical Approximation Strategy}
The numerical approximation that is utilized for simulation of the data assimilation adopts a strategy described in \cite{dg18jomp}, which develops an algorithm for simulating a standard two-phase flow and transport model.
The domain $\Omega$ is discretized into a collection of nonoverlapping rectangles 
$\tau \in \mathcal{T}_h$ such that $\Omega = \bigcup_{\tau\in\mathcal{T}_h} \tau$ (see \cref{fig:disct}), where $h=\max_{\tau\in\mathcal{T}_h} h_\tau$ and $h_\tau$ is the diameter of $\tau$. Let $\mathcal{Z}_h$ be the set of vertices in $\Omega \setminus \Gamma_{\rm D}$ as a result of the partition $\mathcal{T}_h$.
On this $\mathcal{T}_h$, the continuous piecewise bilinear and the discontinuous piecewise bilinear  finite element spaces are respectively defined as
\begin{equation}\label{chpt2_eq:femspace}
	\begin{aligned}
		V_{h} &= \{ v_h\in C(\Omega) \cap H^1_\text{D}:  ~v_h |_{\tau} \text{ is bilinear } \forall \ \tau \in  \mathcal{T}_{h} \}
		= \text{span} \{ \phi_{\bfs \zeta}, ~ {\bfs \zeta} \in \mathcal{Z}_h \},\\
				V_{{\rm d},h} &= \{ v_h\in L^2(\Omega) :  ~v_h |_{\tau} \text{ is bilinear } \forall \ \tau \in  \mathcal{T}_{h} \} =
				\text{span} \{ \phi_{{\bfs \zeta},\tau}, ~{\bfs \zeta} \in \mathcal{Z}_\tau, \tau \in \mathcal{T}_h \},
	\end{aligned}
\end{equation}
where $\phi_{{\bfs \zeta}} : \overline{\Omega} \to [0,1]$ is the usual continuous piecewise linear Lagrangian finite element basis functions and  $\phi_{{\bfs \zeta},\tau} = \phi_{\bfs \zeta} {\bfs 1}_{\tau}$, with ${\bfs 1}_\tau : \overline{\Omega} \to [0,1]$ being the characteristic function of $\tau$. In the above, $\mathcal{Z}_\tau \subset \mathcal{Z}_h$ is the set of vertices in $\tau$. Associated with every vertex $\bfs \zeta \in \mathcal{Z}_h$, denote by $\omega_{\bfs \zeta}$ the control volume whose construction is according to the illustration in \cref{fig:mycv}. The set of such control volumes is denoted by $\mathcal{D}_h$ (see \cref{fig:discd}). 
For $v_h \in V_{{\rm d},h}$, we set $I_\tau v_h$ as a piecewise constant function over $\tau$, which is defined by
$$
I_\tau v_h = \displaystyle \sum_{{\bfs \zeta} \in \mathcal{Z}_\tau} v_h({\bfs \zeta}) {\bfs 1}_{{\bfs \zeta},\tau},
$$
where ${\bfs 1}_{{\bfs \zeta},\tau} : \overline{\Omega} \to [0,1]$ is the characteristic function of the polygonal $\omega_{\bfs \zeta} \cap \tau$ (see Figure~\ref{fig:mycv}).
\begin{figure}
	\centering
	\includegraphics[scale=0.5]{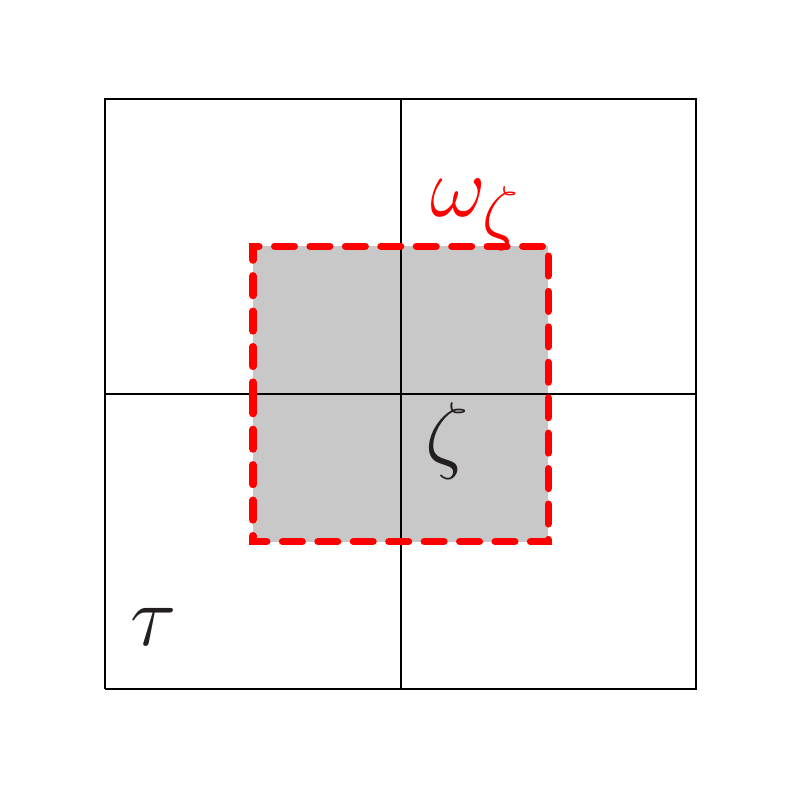}
	\caption{Construction of a control volume $\omega_{\bfs \zeta}$ (the shaded rectangle) for a vertex $\bfs \zeta$ that is shared by four rectangular elements $\tau$. The shaded intersection of a $\tau$ and $\omega_{\bfs \zeta}$ is denoted by $\omega_{\bfs \zeta} \cap \tau$.}
	\label{fig:mycv}
\end{figure}
The semidiscrete approximation of the data assimilation is  to find $(\hat{\theta}_h(t), \hat{p}_h(t), \hat{\Psi}_{{\rm d},h}(t)) \in V_h \times V_h \times V_{{\rm d}, h}$ that is governed by

\begin{subnumcases}{\label{eq:SDDA}}
	\displaystyle \frac{\rm d}{{\rm d} t} \int_{\omega_{\bfs \zeta}}  \hat{\theta}_h  \, \mathrm{d} \boldsymbol{x} - \int_{\partial \omega_{\bfs \zeta}} \left( D \nabla \hat{\theta}_h + \hat{\theta}_{h}  \kappa(\hat{\theta}_h) \nabla \hat{\Psi}_{{\rm d},h}   \right) \cdot \boldsymbol{n} \, \mathrm{d} \ell
	+ \int_{\omega_{\bfs \zeta}} q \hat{\theta}_h  \, \mathrm{d} \boldsymbol{x} \nonumber \\
	\displaystyle \hspace*{0.8cm}= \int_{\omega_{\bfs \zeta}} \left(f - \mu  P_{\hbar}(\hat{\theta}_h -\theta) \right) \, \mathrm{d} \boldsymbol{x} ~ \forall \omega_{\bfs \zeta} \in \mathcal{D}_h, \label{eq:SDDA1}\\
	\displaystyle B(\hat{p}_h,\varphi_h;\hat{\theta}_h) = \langle g, \varphi_h \rangle ~\forall \varphi_h \in V_h, \label{eq:SDDA2}\\
	\displaystyle B_h(\hat{\Psi}_{{\rm d},h},\varphi_{{\rm d},h};\hat{\theta}_h) = R_h(\hat{p}_h, \varphi_{{\rm d},h} ; \hat{\theta}_h) ~\forall \varphi_{{\rm d},h} \in V_{{\rm d},h}, \label{eq:SDDA3}
\end{subnumcases}
where
\begin{align*}
	B_h(v, w ; z) &= \sum_{\tau \in \mathcal{T}_h}  \sum_{{\bfs \zeta} \in \mathcal{Z}_\tau} -\big \langle \kappa(z)  \nabla v  \cdot \boldsymbol{n},  I_\tau w \big \rangle_{\partial \omega_{\bfs \zeta} \cap \tau}, \\
	R_h(v, w ; z) &= \sum_{\tau \in \mathcal{T}_h}\big \langle \{\kappa(z) \nabla v \} \cdot \boldsymbol{n}, I_\tau w - w\big \rangle_{\partial \tau} + \langle g, I_\tau w - w \rangle_\tau + \langle \kappa(z) \nabla v, \nabla w \rangle_\tau,
\end{align*}
and $\boldsymbol{n}$ is the usual outward unit normal vector. In the forms above, for $e = \partial \tau \cap \partial \tau' \subset \partial \tau$, $\{ \boldsymbol{w} \}$ is understood as $(\boldsymbol{w} |_\tau + \boldsymbol{w} |_{\tau'})/2$ on $e$. The formulation \eqref{eq:SDDA} assumes a zero initial condition. 

A closer inspection of the system \eqref{eq:SDDA} shows that it digresses from the original data assimilation model \cref{eq:davarform} in two ways. First, the equation to approximate $\hat{\theta}$, i.e., \eqref{eq:SDDA1} utilizes a finite volume element method, whose variational formulation is different from \eqref{eq:davarform1}. This choice is made due to the method's suitability to approximate solutions to problems that are derived from conservation principle. The second equation \eqref{eq:SDDA2} is exactly the same as in \eqref{eq:davarform2}, except the former is posed on the finite dimensional space $V_h$.

The third equation \eqref{eq:SDDA3} requires a proper description. Due to its global formulation, cf. \eqref{eq:SDDA3}, $\hat{p}_h(t)$ does not satisfy the mass balance property, i.e.,
\begin{equation}
	\int_{\partial \omega} -\kappa(\hat{\theta}_h) \nabla \hat{p}_h \cdot {\bfs n} \, {\rm d} \ell - \int_\omega g \, {\rm d} x \ne 0, ~\forall \, \omega \in \Omega. 
\end{equation}
In many numerical simulations of problems that are derived from conservation principle, it is crucial for the approximations to possess such a mass balance property. The subsystem \eqref{eq:SDDA3} is meant to produce a surrogate of $\hat{p}_h(t) \in V_h$, which is denoted by $\hat{\Psi}_{{\rm d}, h}(t) \in V_{{\rm d},h}$ that satisfies
\begin{equation} \label{eq:lcku}
	\int_{\partial \omega_{{\bfs \zeta}}} -\kappa(\hat{\theta}_h) \nabla \hat{\Psi}_{{\rm d},h}(t) \cdot {\bfs n} \, {\rm d} \ell = \int_{\omega_{{\bfs \zeta}}} g \, {\rm d} x, ~\forall \, \omega_{{\bfs \zeta}} \in \mathcal{D}_h. 
\end{equation}
Thus, \eqref{eq:SDDA3} is a step to produce a locally conservative normal flux $-\kappa (\hat{\theta}_h) \nabla \hat{\Psi}_{{\rm d},h}(t) \cdot {\bfs n}$ in the sense of \cref{eq:lcku}. We refer to \cite{dg18jomp} for an extensive discussion on how this property is established.

Obviously the system in \eqref{eq:SDDA} is coupled and nonlinear so a standard implementation will result in a set of nonlinear algebraic equations governing the following functions:
\begin{equation}
	\begin{aligned}
	\hat{\theta}_h(t) = \sum_{{\bfs \zeta}\in \mathcal{Z}_h} \alpha_{\bfs \zeta}(t) \phi_{\bfs \zeta},~~
    \hat{p}_h(t) = \sum_{{\bfs \zeta} \in \mathcal{Z}_h} \beta_{\bfs \zeta}(t) \phi_{\bfs \zeta},~~\hat{\Psi}_{{\rm d},h}(t) = \sum_{\tau \in \mathcal{T}_h} \sum_{{\bfs \zeta} \in \mathcal{Z}_\tau} \delta_{{\bfs \zeta}, \tau}(t) \phi_{{\bfs \zeta},\tau}.
	\end{aligned}
\end{equation}
The nonlinear algebraic system must be solved by some iterative techniques at every time level. Formal discussions of the construction and techniques for efficiently solving such an algebraic system can be found in many texts (see for example \cite{chen:2006,dg18jomp}). However, there exists a fast-slow relation between the pressure and concentration profiles \cite{aziz:1979}. Taking advantage of this relation can improve the efficiency of the numerical method for solving the above system. This improved efficiency is realized by calculating $\hat{p}_h$ and $\hat{\Psi}_{{\rm d},h}$ on a coarse time partition of $(0,T]$, denoted by $\mathcal{I} = \{t_0, t_1, \cdots, t_M\}$, and computing $\hat{\theta}_h$ at a set of finer time levels $\mathcal{J}_n = \{s_{0,n}, s_{1,n} \cdots, s_{m,n} \}$ that further partitions $(t_{n-1}, t_n]$ for every $n=1,2,\cdots,M$. In this setting, $s_{0,n} = t_{n-1}$ and $s_{m,n} = t_{n}$. Moreover, this separation of time scales allows for a decoupling of the pressure and its surrogate (second and third equations in \eqref{eq:SDDA}) from the concentration (first equation in \eqref{eq:SDDA}). In fact, the second and third equation in \eqref{eq:SDDA}) is also completely decoupled and $\hat{\Psi}_{{\rm d},h}(t_n)$ can be independently solved for every $\tau \in \mathcal{T}_h$ such that a locally conservative normal flux $-\kappa (\hat{\theta}_h(t_{n-1})) \nabla \hat{\Psi}_{{\rm d},h}(t_n) \cdot {\bfs n}$ in the sense of \cref{eq:lcku} is made available. This quantity in turn is fed into equation of the approximate concentration.
In this way, the utilization of iterative techniques to solve \cref{eq:SDDA} can be completely avoided. The complete picture of the time marching on $\mathcal{J}_n$ is listed in Algorithm~\ref{alg:timemarch}. A trapezoidal rule is used to approximate the time integration in that algorithm along with an application of upwinding scheme for the term $\hat{\theta}_h \kappa (\hat{\theta}_h(t_{n-1})) \nabla \hat{\Psi}_{{\rm d},h}(t_n) \cdot \boldsymbol{n}$. This algorithm is then called for every $n=1,2,\cdots,M$.
\begin{algorithm}  
\caption{Time Integration for Data Assimilation on $\mathcal{J}_n$}
\label{alg:timemarch}
\begin{algorithmic}
\State Let $\hat{\theta}_h(t_{n-1}), \hat{p}_h(t_{n-1})$ be available and set $\hat{\theta}_h(s_{0,n}) = \hat{\theta}_h(t_{n-1})$.
\State Set $\kappa_{n-1} = \kappa(\hat{\theta}_h(t_{n-1}))$ and find $\hat{p}_h(t_n) \in V_h$ satisfying
\begin{equation*}
\displaystyle B(\hat{p}_h(t_n),\varphi_h;\hat{\theta}_h(t_{n-1})) = \langle g, \varphi_h \rangle ~\forall \varphi_h \in V_h.
\end{equation*}
\State Find $\hat{\Psi}_{{\rm d},h}(t_n) \in V_{{\rm d},h}$ satisfying
\begin{equation*}
		\displaystyle B_h(\hat{\Psi}_{{\rm d},h}(t_n),\varphi_{{\rm d},h};\hat{\theta}_h(t_{n-1})) = R_h(\hat{p}_h(t_n), \varphi_{{\rm d},h} ; \hat{\theta}_h(t_{n-1})) ~\forall \varphi_{{\rm d},h} \in V_{{\rm d},h},
\end{equation*}
where $\hat{\Psi}_{{\rm d},h}(t_n)$ satisfies the mass balance
\begin{equation*}
	\int_{\partial \omega_{{\bfs \zeta}}} \kappa_{n-1} \nabla \hat{\Psi}_{{\rm d},h}(t_n) \cdot {\bfs n} \, {\rm d} \ell = \int_{\omega_{{\bfs \zeta}}} g \, {\rm d} x, ~\forall \, \omega_{{\bfs \zeta}} \in \mathcal{D}_h. 
\end{equation*}
\For{$j = 1, 2, \cdots, m$}
\smallskip
\State Find $\hat{\theta}_h(s_{j,n}) \in V_h$ satisfying
\begin{equation*}
\begin{aligned}
\displaystyle 
\int_{\omega_{{\bfs \zeta}}}  (\hat{\theta}_h(s_{j,n}) - \hat{\theta}_h(s_{j-1,n}))  \, \mathrm{d} \boldsymbol{x} - \int_{s_{j-1,n}}^{s_{j,n}} \int_{\partial \omega_{{\bfs \zeta}}} \left( D \nabla \hat{\theta}_h + \hat{\theta}_h \kappa_{n-1} \nabla \hat{\Psi}_{{\rm d},h}(t_n) \right) \cdot \boldsymbol{n} \, \mathrm{d} \ell  {\rm d} t\\
\displaystyle \hspace*{0.8cm} + \int_{s_{j-1,n}}^{s_{j,n}} \int_{\omega_{{\bfs \zeta}}} q \hat{\theta}_h  \, \mathrm{d} \boldsymbol{x} {\rm d} t
= \int_{s_{j-1,n}}^{s_{j,n}} \int_{\omega_{{\bfs \zeta}}} \left(f - \mu  P_{\hbar}(\hat{\theta}_h -\theta) \right) \, \mathrm{d} \boldsymbol{x} {\rm d} t ~ \forall \omega_{{\bfs \zeta}} \in \mathcal{D}_h.
\end{aligned}
\end{equation*}
\EndFor
\State Set $\hat{\theta}_h(t_n) = \hat{\theta}_h(s_{m,n})$.  
\end{algorithmic}
\end{algorithm}

The convergence analysis of the proposed algorithm \eqref{eq:SDDA1}-\eqref{eq:SDDA3} is based upon obtaining a priori estimates on the approximation $(\hat{\theta}_h, \hat{p}_h,\hat{\Psi}_{{\rm d},h})$. Similar estimates are obtained  
in Section \ref{sec:Data_Assimilation_Algorithm} for the spectral Galerkin approximation. The actual convergence analysis of the above numerical scheme is far from trivial and is outside the scope of the current investigation. We will disseminate this analysis in a subsequent future work. We refer to \cite{AM-OS} where a different algorithm was used, more specifically a combination of mixed finite element and finite volume methods was implemented and a convergence analysis was proved.

\section{Numerical Validation Study}\label{sec:Examples}
\setcounter{equation}{0}
\setcounter{figure}{0}
The following four examples serve as a numerical validation of the proposed methodology. In the first two examples the relation between the choice of the relaxation parameter, $\mu$, and the corresponding convergence rates of the proposed method are explored by simulating synthetic model problems. In the third example the convergence rates obtained for various choices of the sparse data interpolation length scale, $\hbar$, are explored when an analytic solution to the model problem is entirely unknown. Estimates for the quality of the solutions obtained using the data assimilation algorithm are obtained by comparing to a numerical simulation of the model problem where the initial condition is known. The last example explores the applicability of the methodology for a saltwater intrusion model. Parameter values for the simulation are chosen to respect the physical setting under consideration. Since the previous analysis presented in this paper allows for the choice of $\hat{\theta}_0$ to be made arbitrarily, two types of initial guess are considered. In Examples 1 and 2 the initial guess $\hat{\theta}_0 = 0$ is used. In Examples 3 and 4, a bilinear interpolation of the sparse data $P_\hbar(\theta_0)$ is used as the initial condition for $\hat{\theta}_h$.

\subsection{Example 1}\label{sec:example1}
This example illustrates the validity of \Cref{thm:convergence}. Chiefly, we attempt to showcase the predicted asymptotically stable long run behavior of the error estimates for the assimilated solution. The setting is as follows: $\Omega = [0,1]^2$, $\Gamma_{\text{D}} = \partial \Omega$, $D = 1$, $q = 0$, $-\kappa(\theta) \nabla p = (v(\theta), v(\theta))$, where $v(\theta) = (1+\theta)^{-1}$, $\theta_0(\boldsymbol{x}) = (x_1-{x_1}^2)(x_2-{x_2}^2)$. The function $f$ is chosen such that ${\theta(\boldsymbol{x},t) = (x_1-{x_1}^2)(x_2-{x_2}^2)e^{-t}}$.
 Notice that this example does not require solving for $p$ since the Darcy velocity is directly expressed as a function of $\theta$. The function $\hat{\theta}_h$ is obtained based on the discretization of $\Omega$ into $100 \times 100$ rectangular elements ($h = 0.01$). The interpolation function $P_\hbar$ was constructed using a fixed sparse grid of $10 \times 10$ rectangular elements ($\hbar = 0.1$). A uniform fine time scale of $s_{i+1,n}-s_{i,n} = 0.002$ was chosen with updates to the Darcy velocity occurring at every $10^{\mathrm{th}}$ fine time step. As a metric, the relative difference 
\begin{equation}\label{eq:relativefunc}
 \text{R}(\hat{\theta}_h; t) = 100 \times \frac{\| \hat{\theta}_h(\cdot,t) - \theta(\cdot,t)\|_{2}}{\| \theta(\cdot,t)\|_{2}} \, \%
 \end{equation}
 is calculated at each coarse time step. Plots of values of $\text{R}(\hat{\theta}_h; t)$ for various choices of $\mu$ are shown in \cref{fig:mu1,fig:mu3}. It is worth noting that $\text{R}(\hat{\theta}_h ; 0) = 100 \%$ since the initial condition $\hat{\theta}_0 = 0$ was used in the simulation. 

\begin{figure}[H]\label{figures}
\subfloat[]{\label{fig:mu1} \includegraphics[width=0.32\columnwidth]{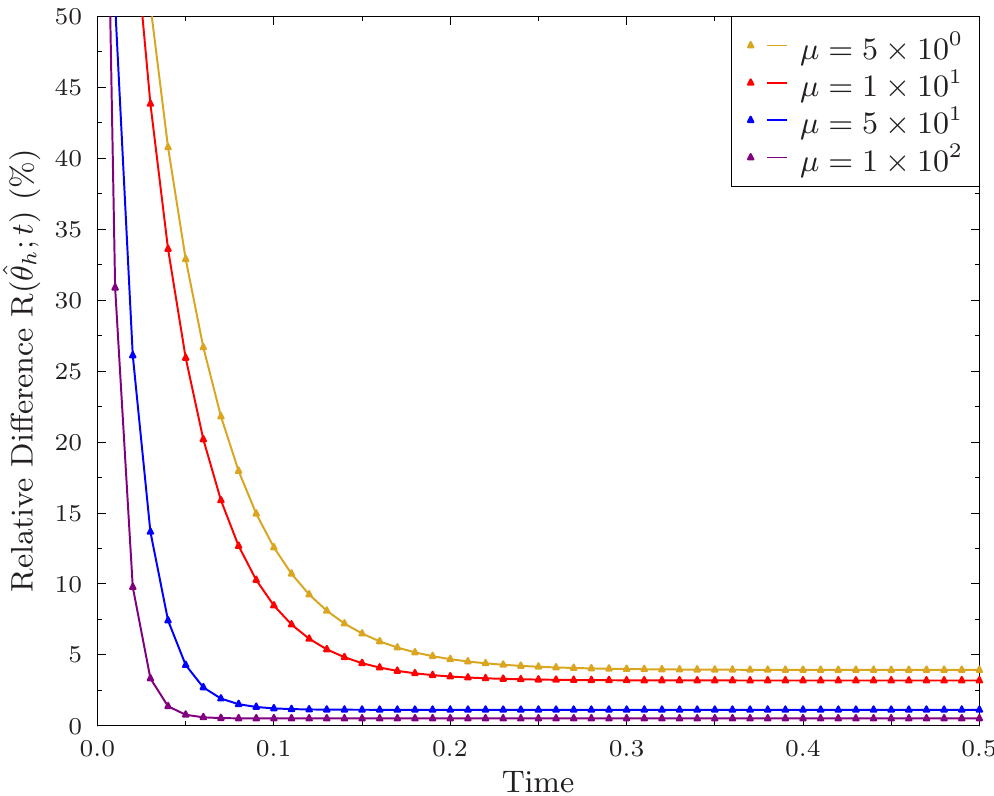}}
\subfloat[]{\label{fig:mu3} \includegraphics[width=0.32\columnwidth]{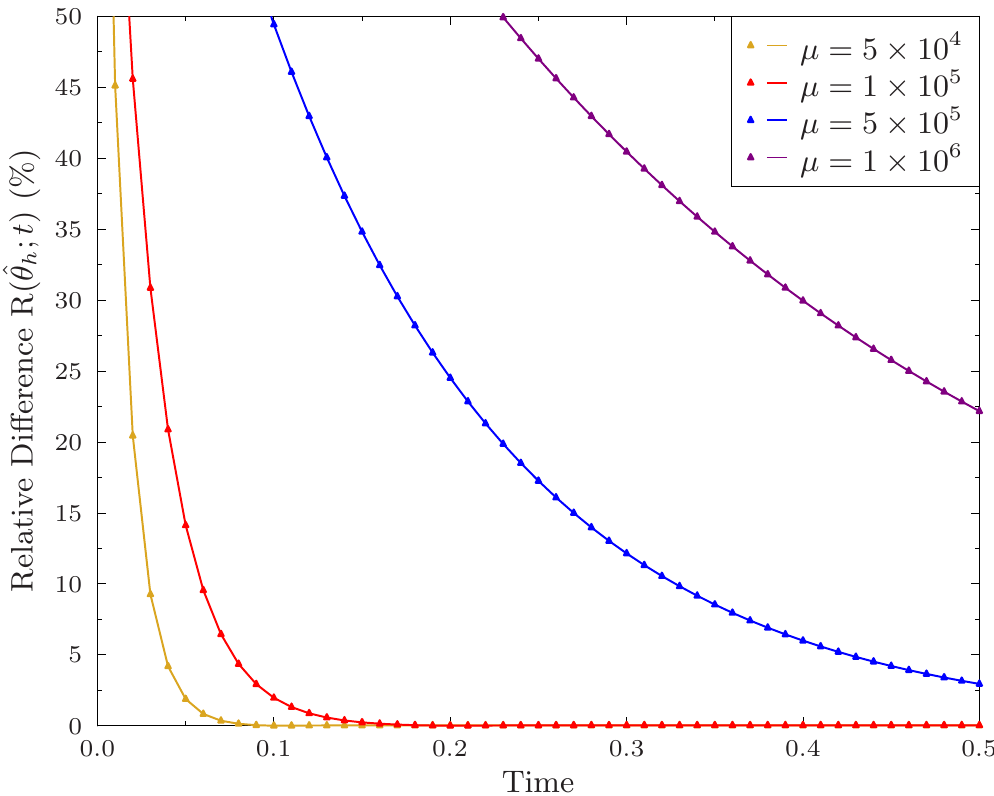}}
\subfloat[]{\label{fig:single} \includegraphics[width=0.32\columnwidth]{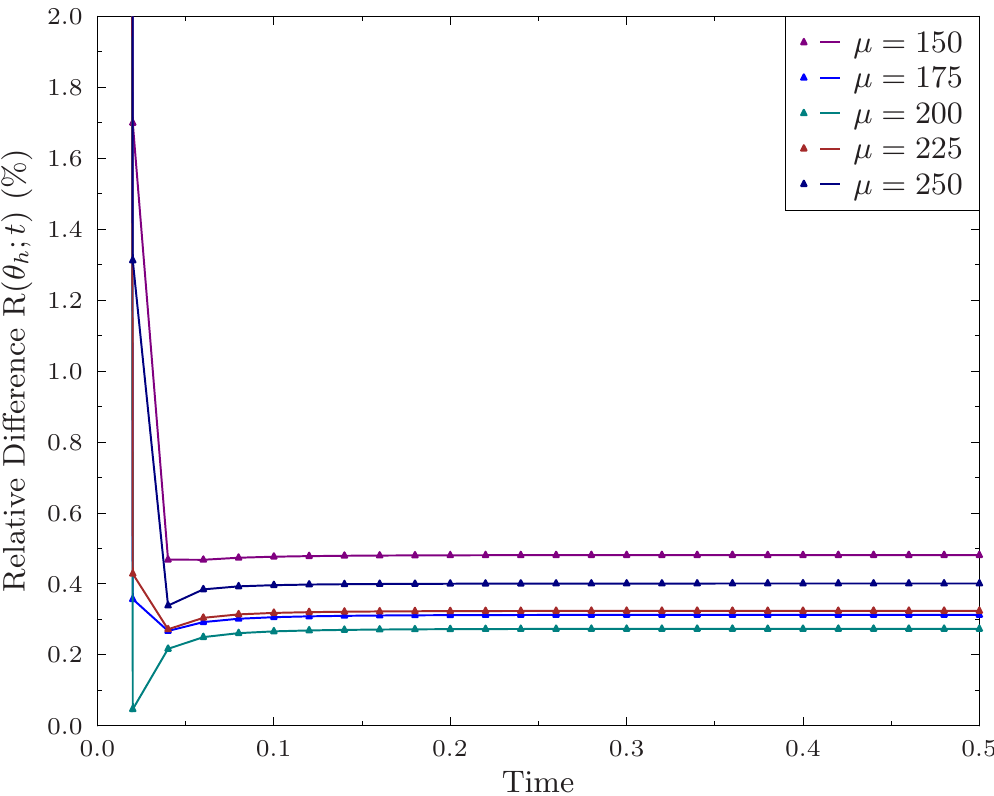}}
\caption{Relative Difference $\text{R}(\hat{\theta}_h;t)$ for various $\mu$: Plots (a) and (b) are for Example 1 with $\hbar = 10 h$, plot (c) is for Example 2 with $\hbar = 5 h$.}
\end{figure}

\begin{figure}[t]
\centering
\subfloat[]{\label{fig:force} \includegraphics[width=0.32\columnwidth]{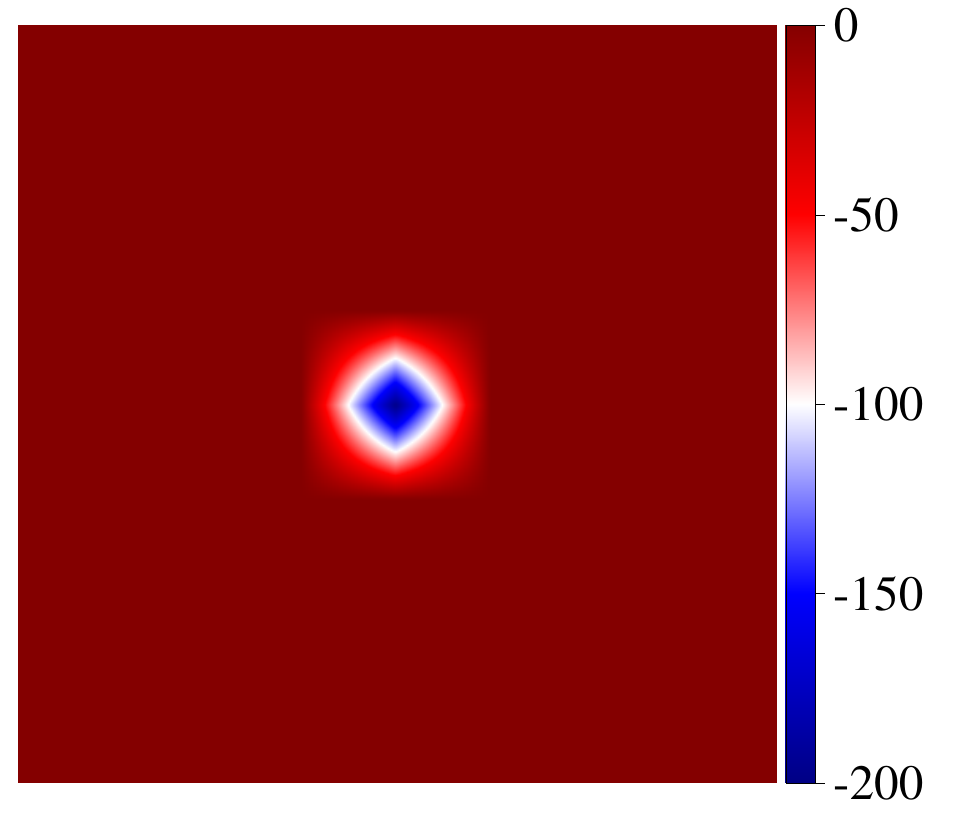}}
\subfloat[]{\label{fig:perm} \includegraphics[width=0.32\columnwidth]{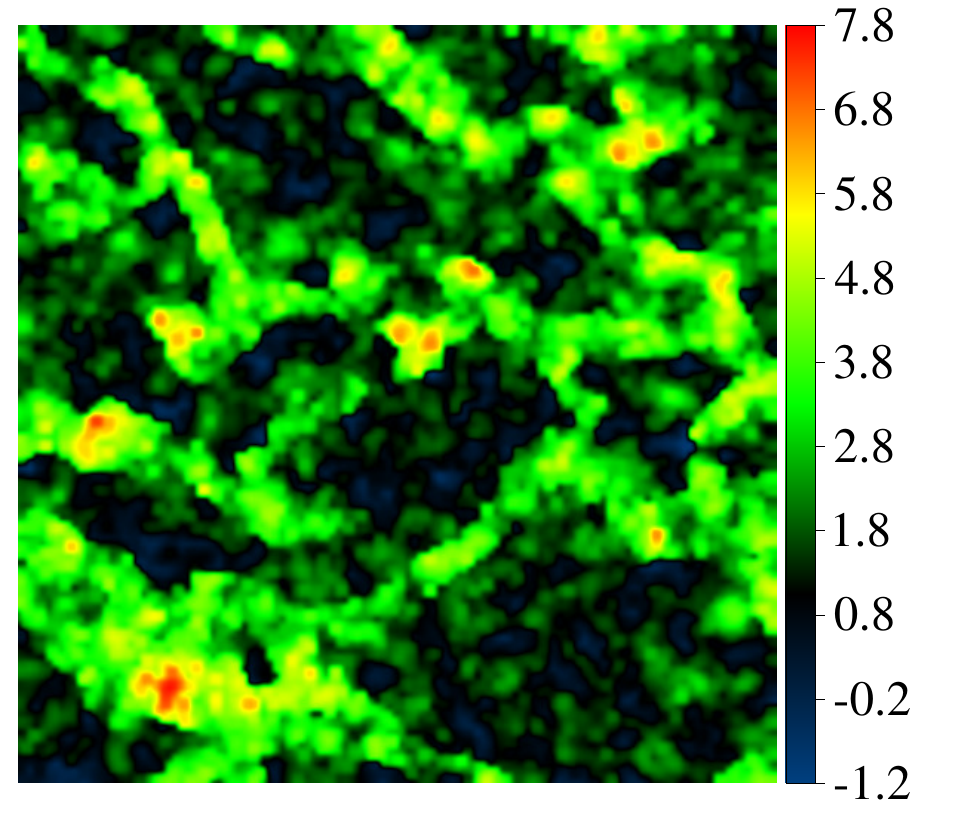}}
\subfloat[]{\label{fig:initcon} \includegraphics[width=0.32\columnwidth]{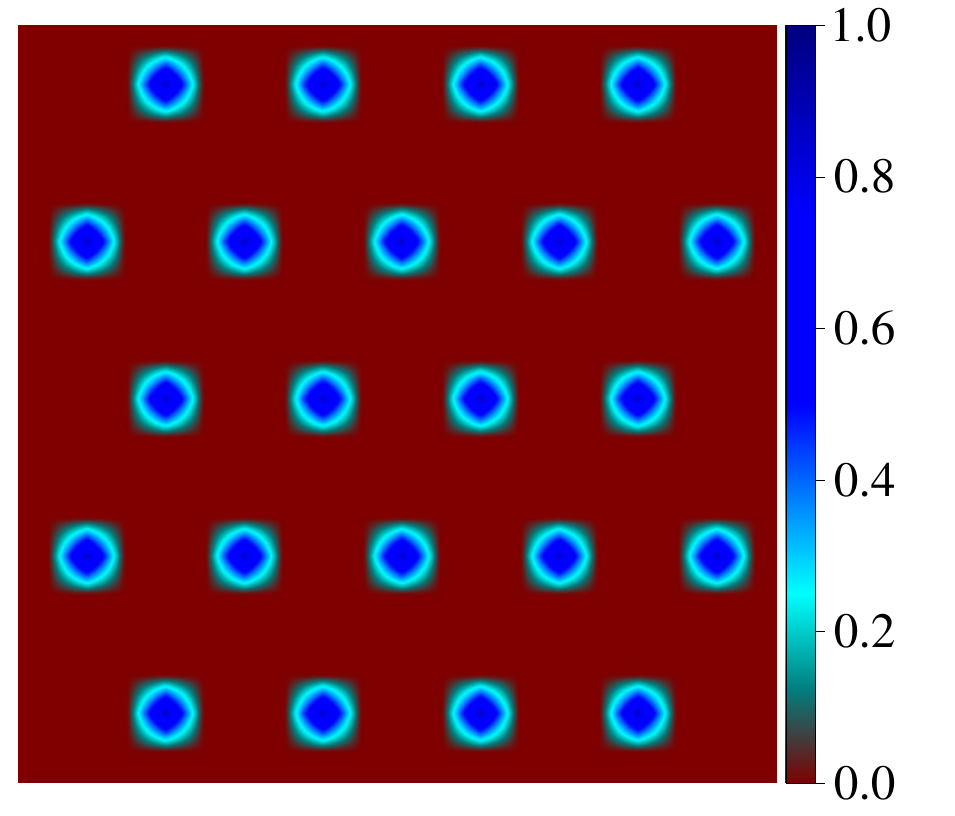}}
\caption{Input data in Example 3 (\Cref{sec:example3}): (a) $g(\boldsymbol{x})$,  (b) $\ln(k(\boldsymbol{x}))$, (c) $\theta_0(\boldsymbol{x})$}
\end{figure}

As a consequence of \Cref{thm:convergence} and when $h$ is sufficiently resolved, an asymptotic bound on $\|\hat{\theta}_h - \theta\|$ can be obtained if conditions on $\mu$ are satisfied. Specifically, combining requirement \cref{eq:muhbarbound} and the condition that $\xi$, defined in \Cref{thm:convergence}, be negative one finds that this asymptotic behavior is obtained if
\begin{equation} \label{eq:perfmu}
2\left(\tilde{C}_{\boldsymbol{w}}+\tilde{C}_{\text{gn}} \right) \,  \| g \|^s - q_*  <  \mu < \frac{D_*}{c_0^2 \hbar}.
\end{equation}
That is, there exists a so-called Goldilocks Zone for choosing the value of $\mu$ to guarantee the desired convergence. The precise value of each of these bounds are unavailable due to lack of information about  the constants of the a priori estimates.   However, the convergence results obtained in \cref{fig:mu1,fig:mu3} are consistent with \Cref{thm:convergence}. Namely, the two figures indicate that there is a specific range of values for $\mu$ for which faster convergence of $\hat{\theta}_h$ to $\theta$ can be achieved. When $\mu$ is chosen within this range, an exponential rate of decrease towards an asymptoticly stable value is observed.  These results also suggest the existence of an optimal value of $\mu$. If the value of $\mu$ is increased or decreased away from this optimal value, the numerical results indicate an increase in the asymptote of the error.

\subsection{Example 2}\label{sec:example2}
This example illustrates the assumption a case of decoupling between $p$ and $\theta$, namely that $\kappa$ is independent of $\theta$. The setting is as follows: $\Omega = [0,1]^2$, $\Gamma_{\text{D}} = (0,x_2) \cup (1,x_2)$, $\Gamma_{\text{N}} = \partial \Omega \backslash \Gamma_{\text{D}}$,  $\kappa = D = 1$, $g = q = 0$, $\theta_0(\boldsymbol{x}) = (x_1-{x_1}^2)$. The function $f$ is chosen such that
 $(p, \theta) = (1-x_1, (x_1 - x_1^2) e^t)$ (i.e., $p(0,x_2) = 1$). The function $\hat{\theta}_h$ is obtained based on the discretization of $\Omega$ into $50 \times 50$ rectangular elements ($h = 0.02$).  The interpolation function $P_\hbar$ was constructed using a fixed sparse grid of $10 \times 10$ rectangular elements ($\hbar = 0.1$).
 A fine time scale of $s_{i+1,m}-s_{i,m} = 0.02$ while the Darcy's velocity is calculated only once before the time marching due the decoupling. Notice that in this decoupling scenario, $\hat{p} = p$.

Plots of $\text{R}(\hat{\theta}_h;t)$ for this example are shown in \cref{fig:single}. Note that for this case Notably, all choices of $\mu$ lead to the convergence of $\hat{\theta}_h$  to $\theta$ in only a few time steps. Furthermore, the resulting asymptotic behavior of the solution can be attributed to the discretization errors of the approximation. Specifically, using the same numerical scheme used to compute $\hat{\theta}_h$ and the known initial condition, an approximation of the true solution was found to have a asymptotic $\text{R}(\theta_h,t)$ value of approximately $0.89\%$. Noting that $\text{R}(\hat{\theta}_h,t)$ converges to an asymptotic value which is less than $0.89\%$ one can conclude that the theoretical convergence of the assimilation procedure has been numerically validated. 

\subsection{Example 3}\label{sec:example3}
The condition \cref{eq:perfmu} suggests an interplay between the choice of $\mu$ and $\hbar$ and the asymptotic behavior of the quality of $\hat{\theta}$. That is, the range of optimal choices for $\mu$ can be adjusted by varying the choice of $\hbar$. We investigate this dependency in the following example by fixing the value of $\mu$ and simulating an example problem using various choices of $\hbar$.

This example is posed in $\Omega = [0,1]^2$ with $\Gamma_{\text{D}} = \partial \Omega$, $D = 0.01$, and $q = f = 0$.
The spatial profile of $g$ is depicted in \cref{fig:force}, and $\kappa(\theta) = k(\boldsymbol{x}) (1-\theta+\theta/16)^{-4}$ where $k(\boldsymbol{x})$ is the heterogeneous profile depicted in log scale in \cref{fig:perm}. The structure of the initial condition $\theta_0$ is shown in \cref{fig:initcon}. The domain $\Omega$ is discretized into $240 \times 240$ rectangular elements. The results are obtained by using $s_{i+1,n}-s_{i,n} = 0.0004$, with the Darcy velocity being updated at a coarse time scale $t_{n+1} - t_n = 0.002$. Sparse data is assumed to be available at each coarse time step, with a bilinear interpolation being performed to update the assimilated concentration at the finest time scale. Since a closed form of the analytic true solution is not available, a numerical solution is used as a reference. That is, we use $\theta = \theta_h$ to obtain estimates of $\text{R}(\hat{\theta}_h,t)$ defined by \cref{eq:relativefunc}. A comparison of the solution obtained with the data assimilation algorithm to the quality of the sparse data used to run the algorithm is obtained using the metric
\begin{equation}
\widetilde{R}(\hat{\theta}_h; t) = 100 \times \frac{\| \hat{\theta}_h(\cdot,t) - P_\hbar\left(\theta_h(\cdot,t)\right)\|_{2}}{\| \theta_h(\cdot,t)\|_{2}} \, \%
\end{equation} 

\begin{figure}[t]
\centering
\begin{minipage}[c]{0.325\columnwidth}
\centering
\subfloat[]{\label{fig:e2rerror}\includegraphics[width=\columnwidth]{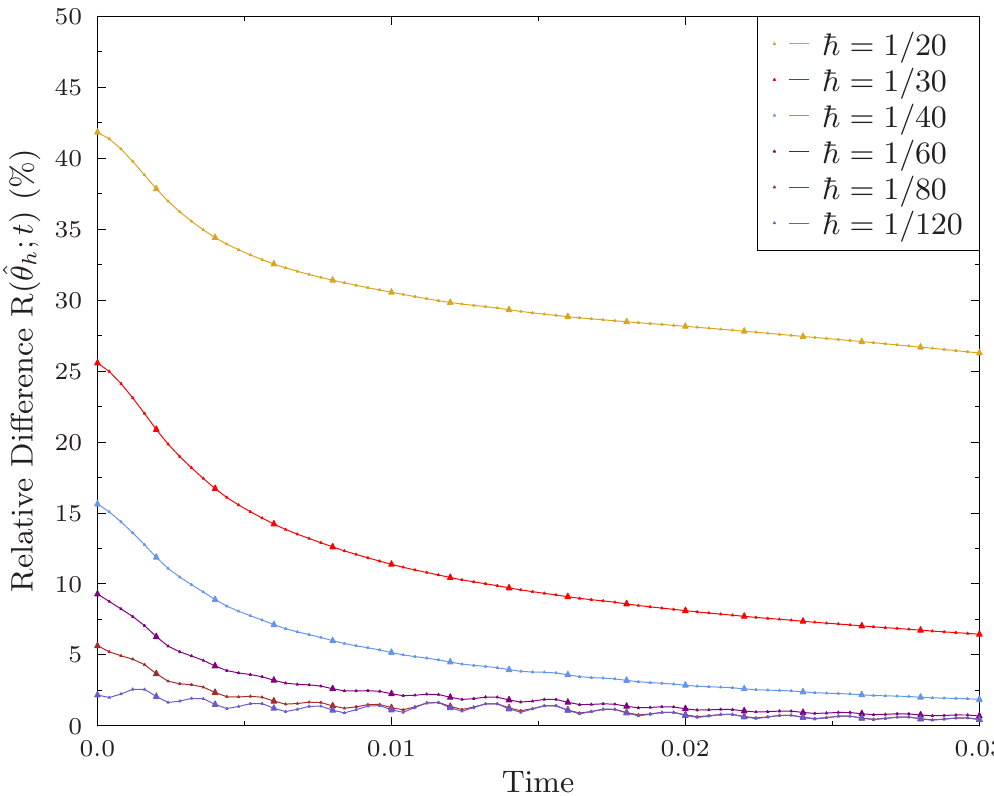}}
\end{minipage}
\begin{minipage}[c]{0.325\columnwidth}
\centering
\subfloat[]{\label{fig:e2ierror}\includegraphics[width=\columnwidth]{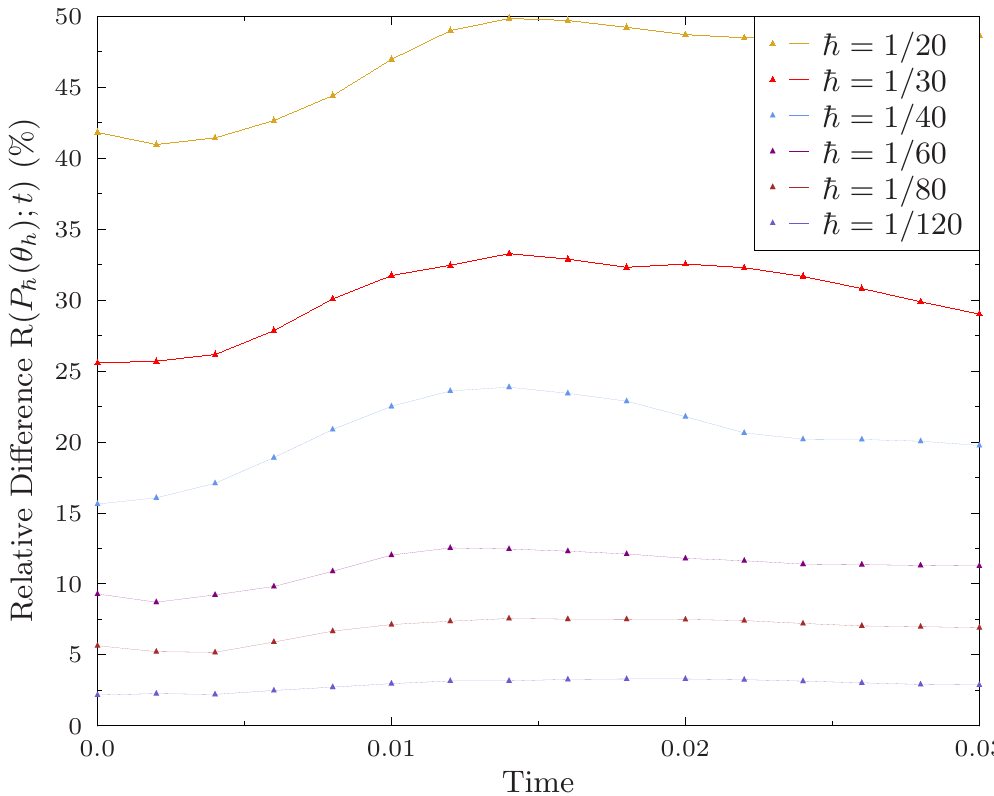}}
\end{minipage}
\begin{minipage}[c]{0.325\columnwidth}
\centering
\subfloat[]{\label{fig:e2ferror}\includegraphics[width=\columnwidth]{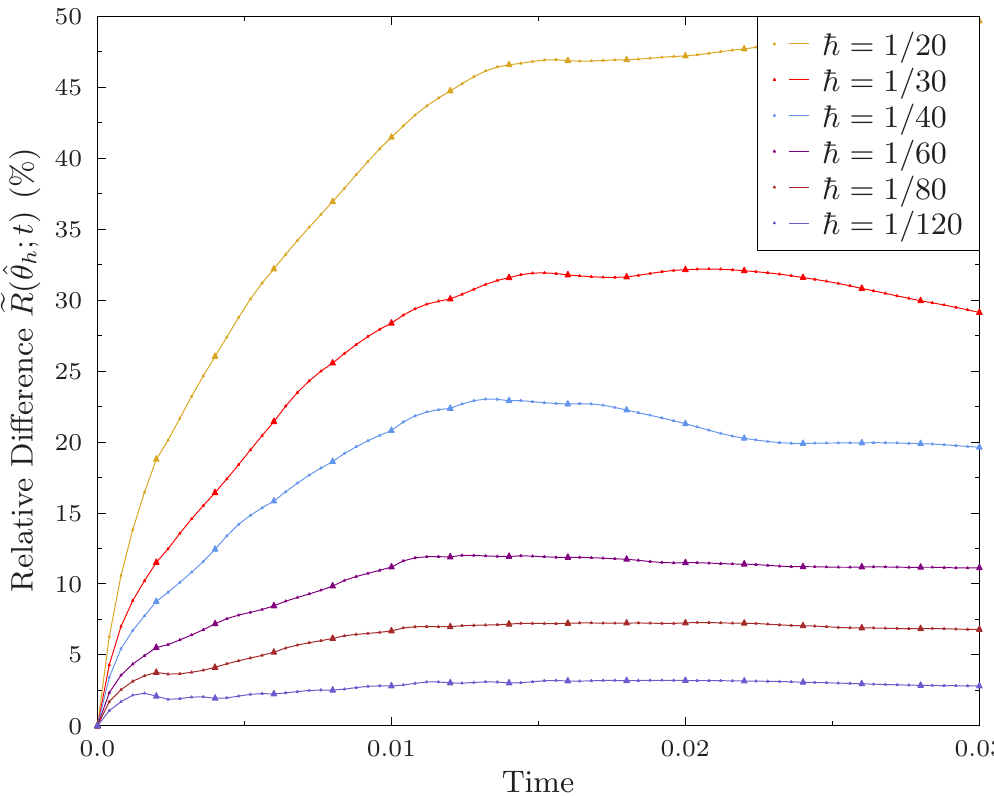}}
\end{minipage}
\caption{Example 3: \cref{fig:e2rerror} shows $\text{R}(\hat{\theta}_h; t)$, the relative difference of the solution obtained using the data assimilation algorithm for various choices of $\hbar$; \cref{fig:e2ierror} shows $\text{R}(P_\hbar(\theta_h); t)$, the quality of a bilinear interpolation of the sparse sample data, at coarse time steps for various choices of $\hbar$; \cref{fig:e2ferror} shows the behavior of relative difference over time, $\tilde{\text{R}}(\hat{\theta}_h; t)$, of the solution obtained using the data assimilation algorithm compared to a bilinear interpolation of the sparse data measurements for various choices of $\hbar$.}
\label{fig:example2}
\end{figure}

\begin{figure}[!t]
\centering
\includegraphics[width=1.0\columnwidth]{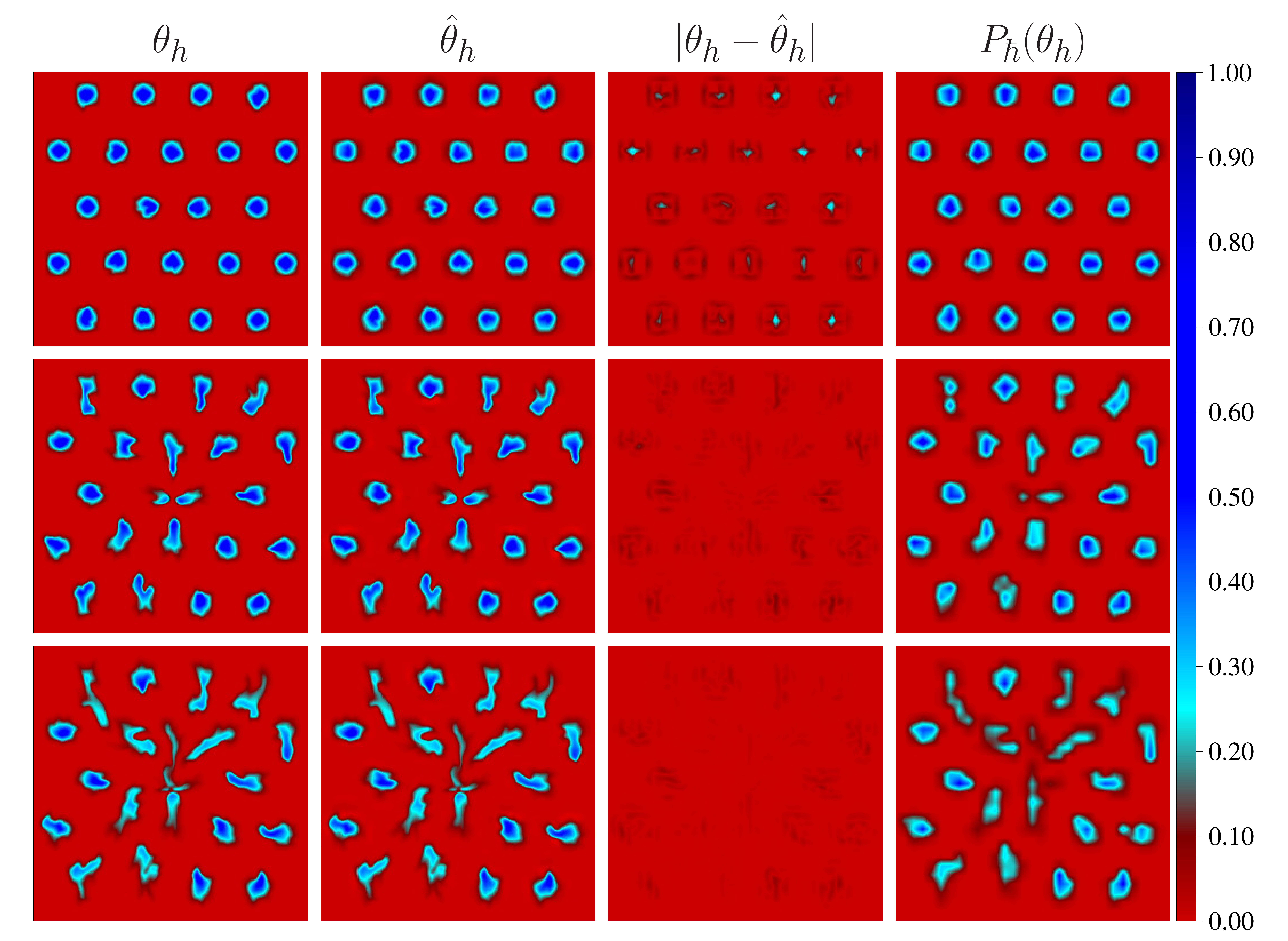}
\caption{Example 3: A comparison of reference concentration against $\hat{\theta}_h$, the magnitude of their difference and a bilinear interpolation of the sparse data, $P_\hbar(\theta_h)$, at time levels $0.002, 0.012$, and $0.024$ (arranged from top to bottom) is given. Values $\hbar = 1/30$ and $\mu = 1000$ were used to generate the profiles presented.}
\label{fig:e2conc}
\end{figure}

In \cref{fig:e2rerror} plots of $\text{R}(\hat{\theta}_h,t)$ for various choices of $\hbar$ are given with respect to time. In all of the examples presented the asymptotic convergence predicted by \Cref{thm:convergence} can be observed. However, there are some subtitles that merit discussion. 

If $\hbar$ is chosen to be too large, the quality of solutions obtained with the data assimilation algorithm is degraded. This result fits with intuition since an increase in the size of $\hbar$ corresponds to a decrease in the number of sparse data measurements used in the simulation. This implies that quality of the assimilated solution may be poor if insufficient information about the true solution is available. Correspondingly, as the choice of $\hbar$ is decreased the accuracy of the data assimilation algorithm is improved. 

As the value of $\hbar$ is decreased small fluctuations in the relative difference at fine time steps can be observed. We suspect that these fluctuations are a result of the interpolation of the sparse data in time. As part of our assumptions for the model we assume that sparse data measurements are available only at each coarse time step and perform a linear interpolation of this data at fine time steps. In \cref{fig:e2rerror} the value of $\text{R}(\hat{\theta}_h,t)$ at coarse time steps is highlighted using larger symbols. After each coarse time step it can be observed that the quality of the assimilated solution begins to improve before becoming less accurate with each fine time step. Essentially, at each coarse time step an update to the sparse data measurements pushes the assimilated solution back towards the true solution.   

By comparing \cref{fig:e2rerror} to \cref{fig:e2ierror,fig:e2ferror} some insight to the relationship between the assimilated solution and the quality of the sparse data used to construct it can be obtained. In \cref{fig:e2ierror} the relative difference between the interpolation of the sparse data, $P_{\hbar}(\theta)$, used to find the assimilated solution and the true solution are given for various choices of $\hbar$. Naturally, as $\hbar$ is decreased (i.e. more interpolation points are used) the interpolation becomes more accurate. Furthermore, if $\hbar$ is on the order of $h$ one observes that the quality of this interpolation and the assimilated solution are similar. However, when the value of $\hbar$ is large one observes that the assimilated solution is significantly more accurate than the bilinear interpolation used to construct it. \Cref{fig:e2ferror} highlights the differences between solutions obtained with the data assimilation algorithm and the interpolation of the sparse data used to run the data assimilation algorithm. Since the $P_\hbar(\theta_0)$ is used as the initial condition for $\hat{\theta}$ the value of $\widetilde{R}(\hat{\theta}_h; 0) = 0$. Then as time increases the difference between the two increases as the assimilated solution better approximates the true model solution. A visual comparison of these differences for a fixed choice of $\hbar$ is provided in \cref{fig:e2conc}.
 
 \begin{figure}[!t]
\centering
\begin{minipage}[c]{0.325\columnwidth}
\centering
\subfloat[]{\label{fig:e3rerror}\includegraphics[width=\columnwidth]{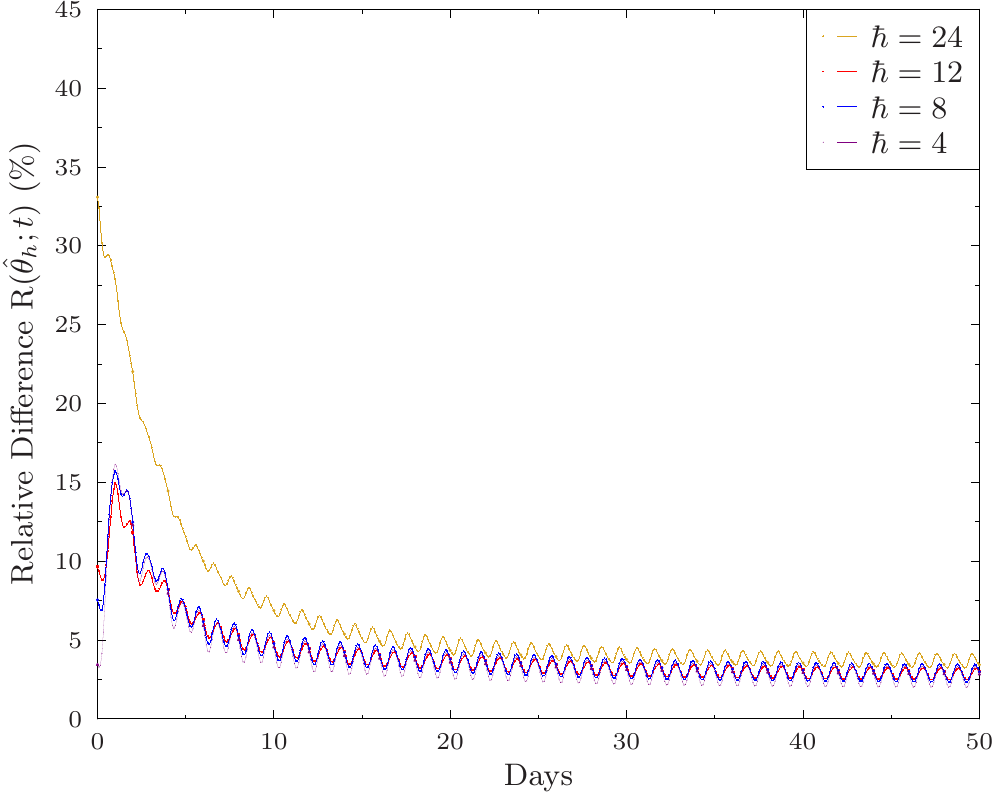}}
\end{minipage}
\begin{minipage}[c]{0.325\columnwidth}
\centering
\subfloat[]{\label{fig:e3ierror}\includegraphics[width=\columnwidth]{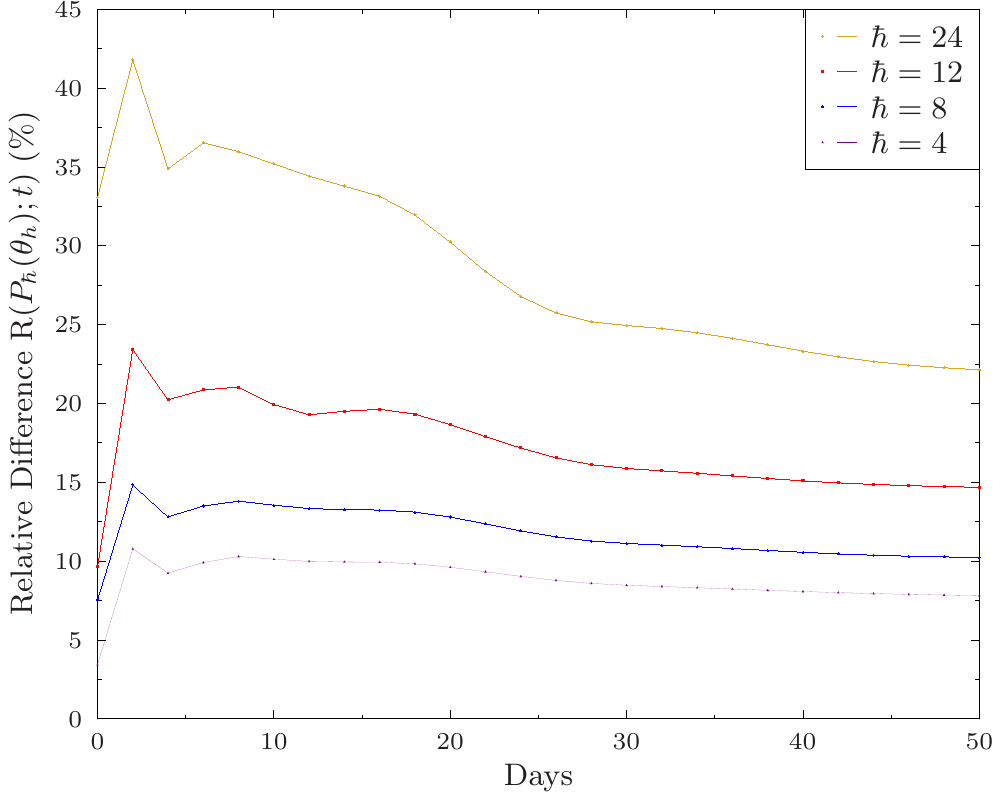}}
\end{minipage}
\begin{minipage}[c]{0.325\columnwidth}
\centering
\subfloat[]{\label{fig:e3ferror}\includegraphics[width=\columnwidth]{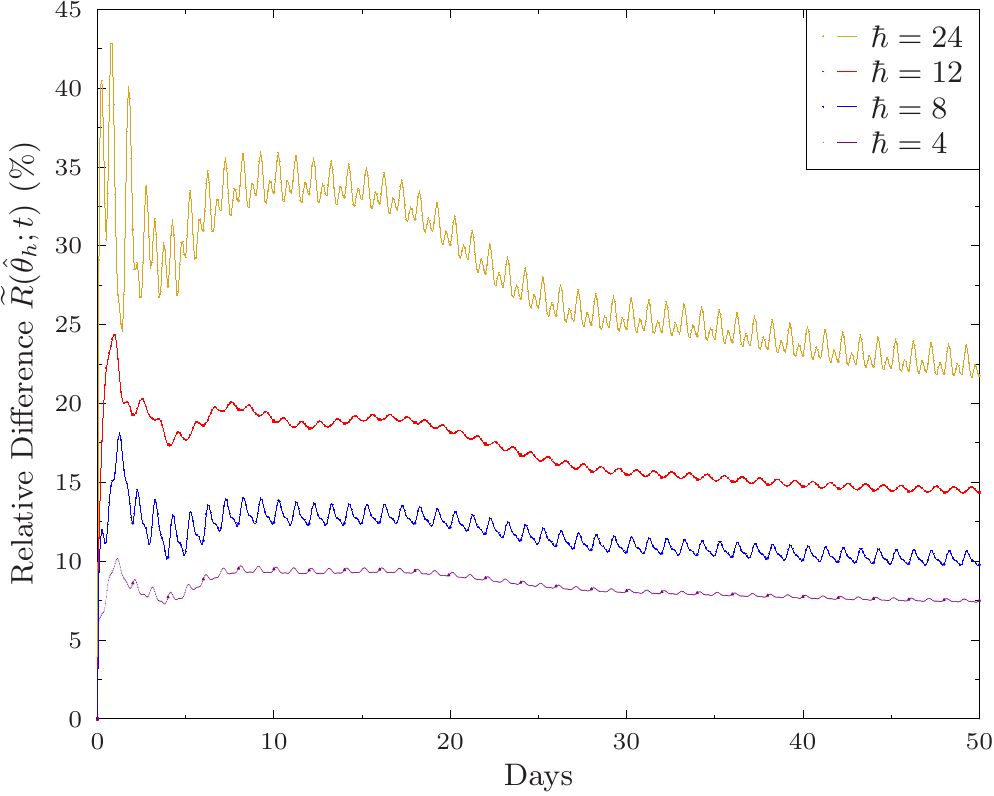}}
\end{minipage}
\caption{Example 4: \cref{fig:e3rerror} shows $\text{R}(\hat{\theta}_h; t)$, the relative difference of the solution obtained using the data assimilation algorithm for various choices of $\hbar$; \cref{fig:e3ierror} shows $\text{R}(P_\hbar(\theta_h); t)$, the quality of a bilinear interpolation of the sparse sample data, at coarse time steps for various choices of $\hbar$; \cref{fig:e3ferror} shows the behavior of relative difference over time, $\tilde{\text{R}}(\hat{\theta}_h; t)$, of the solution obtained using the data assimilation algorithm compared to a bilinear interpolation of the sparse data measurements for various choices of $\hbar$.}
\label{fig:example3}
\end{figure}

\subsection{Example 4}\label{sec:example4}   
In this example we consider an application of the proposed methodology for a model of the intrusion of saltwater into a fresh water aquifer. The mathematical model considered here is governed by \cref{eq:modelproblem}. The problem is posed on a square domain with side length of $240$ meters. The relative permeability is is given by $\kappa(\theta) = k(\boldsymbol{x}) / \nu(\theta)$, where the intrinsic permeability $k(\boldsymbol{x})$ is a scaled version of the profile in \cref{fig:perm} such that $k \in [10^{-9}, 10^{-7}]$. The fluid viscosity $\nu(\theta)$ is modeled with the quarter power mixing rule \cite{1963:Koval}, written as
\begin{equation}
\nu(\theta) = \left(\displaystyle \frac{\theta}{\sqrt[4]{\mu_s}} + \frac{1-\theta}{\sqrt[4]{\mu_w}}\right)^{-4},
\end{equation}
where $\mu_s = 0.00108 \,\text{Pa} \cdot \text{s}$ is the viscosity of the saltwater (or contaminant), and $\mu_w = 0.001 \,\text{Pa} \cdot \text{s}$ is the viscosity of the water (or resident liquid). The diffusion coefficient is $D = 0.00001$. Injection and discharge wells with similar structures to the one used in Example 3 are placed at points $(190,190)$ and $(50,50)$, respectively. The maximum value of $q_\text{in}$ is $0.0005~\text{s}^{-1}$, while for $q_\text{\text{out}}$ is $0.002~\text{s}^{-1}$. At the injection well, $\tilde{\theta}(t) = 0.45 \sin(B t) + 0.5$ where $B$ is chosen so that the period of $\tilde{\theta}(t)$ is one day. The behavior of this injection well is meant to mimic an intrusion of saltwater into the aquifer that depends on the tide of an ocean. Further, initially there are two pockets of saltwater already present in the upper left and lower right hand corners of $\Omega$. Discretization of $\Omega$ utilizes $240 \times 240$ rectangular elements, while the coarse time step is 2 days and fine time step is 2 hours. Updates to the Darcy velocity and the sparse data measurements are made at the end of each coarse time step. As in Example 3, the data assimilation methodology is fed with a numerical solution that is used as for the sparse data measurements and as a reference true solution. 

The same estimates used to evaluate the quality of the assimilated solution in Example 3 are also obtained for this numerical example and are presented in \cref{fig:example3} and \cref{fig:e3conc}. In \cref{fig:example3} a periodicity in the quality of the solution can be observed during the fine time steps. This source of error can be attributed almost entirely to the operator splitting. Since the source term has a period of one day, and each coarse time step has a length of two days the extracted Darcy velocities fail to capture the effect of periodicity of the injection well. When this source of error is taken into account many of the same comparisons and conclusions discussed in Example 3 can be made for for this example as well.

\begin{figure}[!t]
\centering
\includegraphics[width=1.0\columnwidth]{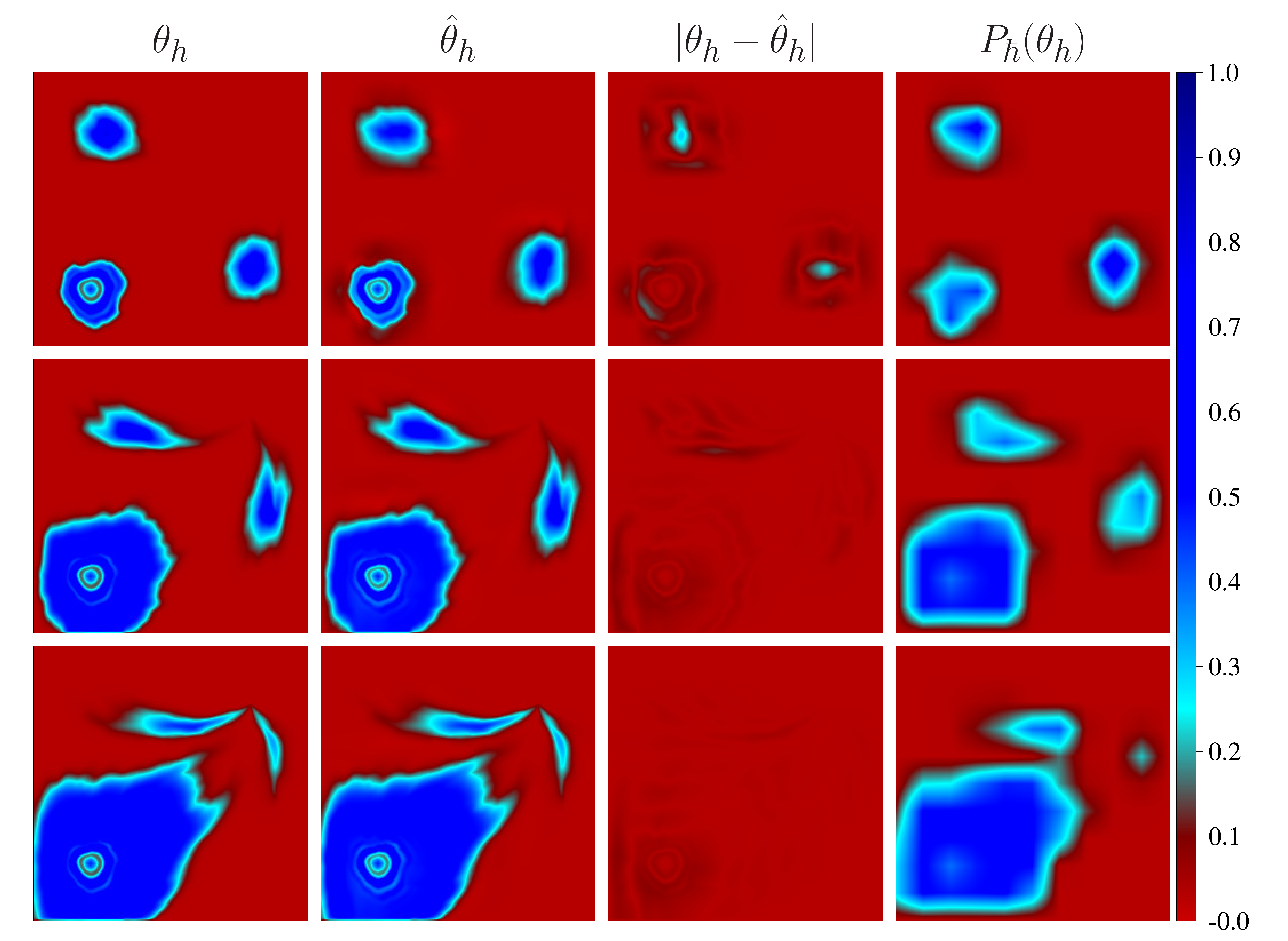}
\caption{Example 4: A comparison of reference concentration against $\hat{\theta}_h$, the magnitude of their difference and a bilinear interpolation of the sparse data, $P_\hbar(\theta_h)$ at time levels $0.002$, $0.012$, and $0.024$ (arranged from top to bottom). Values $\hbar = 40$ m and $\mu = 0.00001$ were used to generate the profiles.}
\label{fig:e3conc}
\end{figure}

\section{Conclusion}\label{sec:conclusion}
\setcounter{equation}{0}
\setcounter{figure}{0}
In this paper we have adapted the framework of a continuous data assimilation algorithm for dynamical systems to a model for the miscible flow of two fluids in a porous medium. The novelty of the proposed methodology is that it provides a way for obtaining accurate approximations of the true solution when nothing is known about the initial conditions of the system. Specifically, in situations where it is reasonable to assume that sparse spatial measurements of the concentration of either fluid can periodically be measured at sparse spatial locations this method can be used. A formal analysis of both the existence of weak solutions to the data assimilation model, and estimates for its associated error were performed. A series of numerical simulations was conducted to validate the conclusion of the mathematical analysis. It was demonstrated through these numerical studies that there exists a range of values for the relaxation parameter which can lead to a desirable rate of convergence. This range of values can be be further adjusted by adjusting the length scale of the sparse data measurements.

In our numerical experiments it can be observed that there exists an optimal choice for this parameter within its bounds. A direction for possible future work would be the development of a systematic way to determine a priori estimates for this optimal value. Furthermore, an introduction of stochastic noise into the sparse data measurements could be used to quantify the effect that measurement errors would have on the assimilated solution. Lastly, in the model considered it is assumed that the diffusion coefficient is only spatially dependent. More realistic models of miscible flow assume that the diffusion coefficient is also velocity dependent. It is suspected that similar estimates to the ones found here can be obtained for the more general model. However, a formal analysis must be performed to determine the convergence behavior of the continuous data assimilation algorithm in this setting.

\bibliographystyle{plain}
\bibliography{references}

\begin{thebibliography}{10}

\bibitem{200379}
Robert~A. Adams and John~J.F. Fournier.
\newblock {\em Sobolev Spaces}, volume 140 of {\em Pure and Applied
  Mathematics}.
\newblock Elsevier, 2003.

\bibitem{AM-OS}
Brahim Amaziane and Mustapha El~Ossmani.
\newblock Convergence analysis of an approximation to miscible fluid flows in
  porous media by combining mixed finite element and finite volume methods.
\newblock {\em Numerical Methods for Partial Differential Equations},
  24(3):799--832, 2008.

\bibitem{aziz:1979}
K.~Aziz, K.~Aziz, and A.~Settari.
\newblock {\em Petroleum Reservoir Simulation}.
\newblock Applied Science Publishers, 1979.

\bibitem{2013:Titi}
Abderrahim Azouani, Eric Olson, and Edriss Titi.
\newblock Continuous data assimilation using general interpolant observables.
\newblock {\em J. Nonlinear Sci.}, 24(2):277--304, 2014.

\bibitem{1972:bear}
J.~Bear.
\newblock {\em Dynamics of Fluids in Porous Media}.
\newblock Dover Civil and Mechanical Engineering Series. Dover, 1972.

\bibitem{2010:Bear}
J.~Bear and A.H.D. Cheng.
\newblock {\em Modeling Groundwater Flow and Contaminant Transport}.
\newblock Theory and Applications of Transport in Porous Media. Springer
  Netherlands, 2010.

\bibitem{MR2852305}
Jean-Marie~Emmanuel Bernard.
\newblock Density results in {S}obolev spaces whose elements vanish on a part
  of the boundary.
\newblock {\em Chin. Ann. Math. Ser. B}, 32(6):823--846, 2011.

\bibitem{MR3078113}
D.~Bl\"{o}mker, K.~Law, A.~M. Stuart, and K.~C. Zygalakis.
\newblock Accuracy and stability of the continuous-time 3{DVAR} filter for the
  {N}avier-{S}tokes equation.
\newblock {\em Nonlinearity}, 26(8):2193--2219, 2013.

\bibitem{MR1884727}
Adel Blouza and Herv\'{e} Le~Dret.
\newblock An up-to-the-boundary version of {F}riedrichs's lemma and
  applications to the linear {K}oiter shell model.
\newblock {\em SIAM J. Math. Anal.}, 33(4):877--895, 2001.

\bibitem{2008:Brenner}
S.~C. Brenner and L.~R. Scott.
\newblock {\em The Mathematical Theory of Finite Element Methods}, volume~15 of
  {\em Texts in Applied Mathematics}.
\newblock Springer, New York, third edition, 2008.

\bibitem{MR3008183}
C.~E.~A. Brett, K.~F. Lam, K.~J.~H. Law, D.~S. McCormick, M.~R. Scott, and
  A.~M. Stuart.
\newblock Accuracy and stability of filters for dissipative {PDE}s.
\newblock {\em Phys. D}, 245:34--45, 2013.

\bibitem{brezis2010functional}
H.~Brezis.
\newblock {\em Functional Analysis, Sobolev Spaces and Partial Differential
  Equations}.
\newblock Universitext. Springer New York, 2010.

\bibitem{1999:Chen}
Z.~Chen and R.~Ewing.
\newblock Mathematical analysis for reservoir models.
\newblock {\em SIAM J. Math. Anal.}, 30(2):431--453, 1999.

\bibitem{chen:2006}
Z.~Chen, G.~Huan, and Y.~Ma.
\newblock {\em Computational Methods for Multiphase Flows in Porous Media},
  volume~2 of {\em Computational Science \& Engineering}.
\newblock Society for Industrial and Applied Mathematics (SIAM), Philadelphia,
  PA, 2006.

\bibitem{MR1205006}
G.~Ciccarella, M.~Dalla~Mora, and A.~Germani.
\newblock A {L}uenberger-like observer for nonlinear systems.
\newblock {\em Internat. J. Control}, 57(3):537--556, 1993.

\bibitem{DeBoor:1428148}
Carl De~Boor.
\newblock {\em {A practical guide to splines; rev. ed.}}
\newblock Applied mathematical sciences. Springer, Berlin, 2001.

\bibitem{dg18jomp}
Quanling Deng and Victor Ginting.
\newblock Locally conservative continuous galerkin fem for pressure equation in
  two-phase flow model in subsurfaces.
\newblock {\em Journal of Scientific Computing}, 74(3):1264--1285, 2018.

\bibitem{MR2261637}
Pavel Doktor and Alexander \v{Z}en\'{\i}\v{s}ek.
\newblock The density of infinitely differentiable functions in {S}obolev
  spaces with mixed boundary conditions.
\newblock {\em Appl. Math.}, 51(5):517--547, 2006.

\bibitem{2014:Talbot}
J.~Droniou and K.~Talbot.
\newblock On a miscible displacement model in porous media flow with measure
  data.
\newblock {\em SIAM J. Math. Anal.}, 46(5):3158--3175, 2014.

\bibitem{2000:Fabrie}
P.~Fabrie and T.~Gallouet.
\newblock Modeling wells in porous media flows.
\newblock {\em Mathematical Models and Methods in Applied Sciences}, volume 10,
  numero 5:673--709, 2000.

\bibitem{MR1185634}
Pierre Fabrie and Michel Langlais.
\newblock Mathematical analysis of miscible displacement in porous medium.
\newblock {\em SIAM J. Math. Anal.}, 23(6):1375--1392, 1992.

\bibitem{2015:Farhat}
A.~Farhat, M.~S. Jolly, and E.~Titi.
\newblock Continuous data assimilation for the 2{D} {B}\'enard convection
  through velocity measurements alone.
\newblock {\em Phys. D}, 303:59--66, 2015.

\bibitem{MR1789483}
T.~Gallouet and A.~Monier.
\newblock On the regularity of solutions to elliptic equations.
\newblock {\em Rend. Mat. Appl. (7)}, 19(4):471--488 (2000), 1999.

\bibitem{Gesho}
M.~Gesho, E.~Olson, and E.~Titi.
\newblock A computational study of a data assimilation algorithm for the
  two-dimensional {N}avier-{S}tokes equations.
\newblock {\em Commun. Comput. Phys.}, 19(4):1094--1110, 2016.

\bibitem{MR990595}
Konrad Gr\"{o}ger.
\newblock A {$W^{1,p}$}-estimate for solutions to mixed boundary value problems
  for second order elliptic differential equations.
\newblock {\em Math. Ann.}, 283(4):679--687, 1989.

\bibitem{hollig03}
Klaus Höllig.
\newblock {\em Finite Element Methods with B-Splines}.
\newblock Society for Industrial and Applied Mathematics, 2003.

\bibitem{1963:Koval}
E.~J. Koval.
\newblock A method for predicting the performance of unstable miscible
  displacement in heterogeneous media.
\newblock {\em Society of Petroleum Engineers Journal}, 3(02):145--154, 1963.

\bibitem{lorenc86}
A.~C. Lorenc.
\newblock Analysis methods for numerical weather prediction.
\newblock {\em Quarterly Journal of the Royal Meteorological Society},
  112(474):1177--1194, 1986.

\bibitem{MR3363684}
David~G. Luenberger and Yinyu Ye.
\newblock {\em Linear and nonlinear programming}, volume 228 of {\em
  International Series in Operations Research \& Management Science}.
\newblock Springer, Cham, fourth edition, 2016.

\bibitem{2002:Majada}
A.~Majda and A.~Bertozzi.
\newblock {\em Vorticity and incompressible flow}, volume~27 of {\em Cambridge
  Texts in Applied Mathematics}.
\newblock Cambridge University Press, Cambridge, 2002.

\bibitem{2013:McCormick}
D.~McCormick, J.~Robinson, and J.~Rodrigo.
\newblock Generalised {G}agliardo--{N}irenberg inequalities using weak
  {L}ebesgue spaces and {BMO}.
\newblock {\em Milan Journal of Mathematics}, 81(2):265--289, 2013.

\bibitem{2017:Meiss}
J.~D. Meiss.
\newblock {\em Differential dynamical systems}, volume~22 of {\em Mathematical
  Modeling and Computation}.
\newblock Society for Industrial and Applied Mathematics (SIAM), Philadelphia,
  PA, revised edition, 2017.

\bibitem{MR159110}
Norman~G. Meyers.
\newblock An {$L^{p}$}e-estimate for the gradient of solutions of second order
  elliptic divergence equations.
\newblock {\em Ann. Scuola Norm. Sup. Pisa Cl. Sci. (3)}, 17:189--206, 1963.

\bibitem{1991:Mikelic}
A.~Mikeli\'c.
\newblock Mathematical theory of stationary miscible filtration.
\newblock {\em J. Differential Equations}, 90(1):186--202, 1991.

\bibitem{nirenberg59}
Louis Nirenberg.
\newblock On elliptic partial differential equations.
\newblock {\em Annali della Scuola Normale Superiore di Pisa - Classe di
  Scienze}, Ser. 3, 13(2):115--162, 1959.

\bibitem{1962:Peaceman}
D.~Peaceman and H.~Rachford.
\newblock Numerical calculation of multidimensional miscible displacement.
\newblock {\em Old SPE Journal}, 2(4):327--339, 1962.

\bibitem{1101300}
{Tzyh-Jong Tarn} and Y.~{Rasis}.
\newblock Observers for nonlinear stochastic systems.
\newblock {\em IEEE Transactions on Automatic Control}, 21(4):441--448, 1976.

\bibitem{1988:Vishik}
M.~Vishik and A.~Fursikov.
\newblock {\em Mathematical problems of statistical hydromechanics}, volume~9
  of {\em Mathematics and its Applications (Soviet Series)}.
\newblock Kluwer Academic Publishers Group, Dordrecht, 1988.
\newblock Translated from the 1980 Russian original by D. A. Leites.

\end{thebibliography}

\appendix
\setcounter{equation}{0}
\section{\texorpdfstring{$L^{\infty}$ Bounds on the True Concentration}{TEXT}}\label{sec:appendix1}
We begin by assuming that \cref{as:kappa,as:dif,as:g,as:f,as:q,as:init} are satisfied and that there exists a $\theta \in L^{\infty}_T (L^2(\Omega)) \cap L_T^{2} (H^1_{\text{D}})$ that satisfies \cref{eq:genvarform}. We write $\theta = \theta^+ - \theta^-$, where $\theta^{+} = \text{max}\{0,\theta\}$ and $\theta^{-} = \text{max}\{0,-\theta\}$ and
$$
\nabla \theta^+ =
\begin{cases}
\nabla \theta  &\text{ if } \theta>0,\\
0 & \text{ if }  \theta\le 0,
\end{cases}
\text{ and }
\nabla \theta^- =
\begin{cases}
0 &\text{ if } \theta>0,\\
-\nabla \theta & \text{ if }  \theta\le 0.
\end{cases}
$$

Since
$$\langle \theta^+ - \theta^-, -\theta^- \rangle = \| \theta^- \|^2
\text{ and }
\langle D\nabla (\theta^+ - \theta^-), -\nabla \theta^- \rangle = \langle D \nabla \theta^-, \nabla \theta^- \rangle,
$$
we may use $-\theta^{-}(\cdot, t)$ as a test function in \eqref{eq:subnuma} to get
\begin{equation}\label{eq:infbdtest}
 \frac{1}{2} \frac{\mathrm{d} }{\mathrm{d} t} \| \theta^- \|_{2}^2 +A(\theta^-, p,\theta^-; \theta) = -\langle f, \theta^- \rangle.
\end{equation} 
As a consequence of the positivity of $D$ outlined in \cref{as:dif}, \cref{eq:infbdtest} yields
\begin{equation}\label{eq:infbdtest2}
 \frac{1}{2} \frac{\mathrm{d} }{\mathrm{d} t} \| \theta^- \|_{2}^2  
 + \langle \theta^{-} \kappa(\theta) \nabla p , \nabla \theta^{-} \rangle
 + \langle f +  q \theta^{-},\theta^- \rangle  \le 0.
\end{equation}

Let now $\{ \theta_\ell \} \subset H^1_\text{D}$ be smooth and $\| \theta_\ell - \theta^- \|_1 \to 0$ as $\ell \to \infty$. Recent investigations confirming the existence of such a sequence can be seen for example in \cite{MR1884727,MR2261637,MR2852305}. By Cauchy-Schwarz inequality,
\begin{equation} \label{eq:A3}
\begin{aligned}
\big  | \langle \theta_\ell \kappa(\theta) \nabla p, \nabla \theta_\ell \rangle - 
\langle \theta^- \kappa(\theta) \nabla p, \nabla \theta^- \rangle \big |
&= \big  | \langle (\theta_\ell - \theta^-) \kappa(\theta) \nabla p, \nabla \theta_\ell \rangle + 
\langle \theta^- \kappa(\theta) \nabla p, \nabla (\theta_\ell - \theta^-) \rangle \big |\\
&\le \kappa^* \left( \| (\theta_\ell - \theta^-) \nabla p \| \, \| \nabla \theta_\ell \| +
\| \theta^-  \nabla p \| \, \| \nabla (\theta_\ell - \theta^-) \| \right).
\end{aligned}
\end{equation}
By H\"older's inequality and applying Meyer's estimate to $p$ and Sobolev embedding $H^1(\Omega) \hookrightarrow L^s(\Omega)$ for $s>2$ (see for example p. 85 of \cite{200379}),
\begin{equation} \label{eq:A4}
\begin{aligned}
&\| (\theta_\ell - \theta^-) \nabla p \| \le \| \theta_\ell - \theta^- \|_{0,s} \, \| \nabla p \|_{0,r} \le 
C_\text{emb} C(r) \| g \| \, \| \theta_\ell - \theta^- \|_1, \text{ and }\\
& \| \theta^-  \nabla p \| \le \| \theta^- \|_{0,s} \, \| \nabla p \|_{0,r} \le 
C_\text{emb} C(r) \| g \| \, \| \theta^- \|_1, \text{ where } \frac{1}{r} + \frac{1}{s} = \frac{1}{2}.
\end{aligned}
\end{equation}
Putting \cref{eq:A4} back to \cref{eq:A3} yields
\begin{equation}
\big  | \langle \theta_\ell \kappa(\theta) \nabla p, \nabla \theta_\ell \rangle - 
\langle \theta^- \kappa(\theta) \nabla p, \nabla \theta^- \rangle \big | \le
\kappa^* C_\text{emb} C(r) \| g \|  \, \| \theta_\ell \|_1 \, \| \theta_\ell - \theta^- \|_1,
\end{equation}
from which it is deduced that
\begin{equation} \label{eq:A6}
\big  | \langle \theta_\ell \kappa(\theta) \nabla p, \nabla \theta_\ell \rangle - 
\langle \theta^- \kappa(\theta) \nabla p, \nabla \theta^- \rangle \big | \to 0 \text{ as } \ell \to \infty.
\end{equation}
Moreover, due to its smoothness, $\nabla (\theta_\ell^2) = 2 \theta_\ell \nabla \theta_\ell$ and since $p$ satisfies \eqref{eq:subnumb}, one finds that 
\begin{equation}\label{eq:advecboundplus}
\langle \theta_\ell \kappa({\theta}) \nabla {p},  \nabla \theta_\ell \rangle =\frac{1}{2} \langle \kappa({\theta}) \nabla {p} , \nabla (\theta_\ell^2) \rangle = \frac{1}{2} \langle g, (\theta_\ell)^2 \rangle,
\end{equation}
where the last line is justified because $\theta_\ell^2 \in H^1_\text{D}$, which is the case  due to the smoothness of $\theta_\ell \in H^1_\text{D}$. Furthermore, by Cauchy-Schwarz inequality and Sobolev embedding $H^1(\Omega) \hookrightarrow L^s(\Omega)$ for $s>2$, 
\begin{equation}
\begin{aligned}
| \langle g, (\theta_\ell)^2 \rangle - \langle g, (\theta^-)^2 \rangle | &= 
| \langle g, (\theta_\ell + \theta^-)(\theta_\ell - \theta^-) \rangle  |\\
&\le \| g \| \, \| (\theta_\ell + \theta^-)(\theta_\ell - \theta^-) \|\\
&\le \| g \| \, \| (\theta_\ell + \theta^-) \|_{0,\zeta} \, \|\theta_\ell - \theta^- \|_{0,\nu}, \text{ with } \frac{1}{\zeta} + \frac{1}{\nu} = \frac{1}{2}\\
&\le C^2_\text{emb} \| g \| \, \| (\theta_\ell + \theta^-) \|_1 \, \|\theta_\ell - \theta^- \|_1,
\end{aligned}
\end{equation}
from which it is deduced that
\begin{equation} \label{eq:A9}
| \langle g, (\theta_\ell)^2 \rangle - \langle g, (\theta^-)^2 \rangle | \to 0 \text{ as } \ell \to \infty.
\end{equation}
Combination of \cref{eq:A6}, \cref{eq:advecboundplus}, and \cref{eq:A9} yields
\begin{equation} \label{eq:A10}
\langle \theta^{-} \kappa(\theta) \nabla p , \nabla \theta^{-} \rangle = \lim_{\ell \to \infty}
\langle \theta_\ell \kappa(\theta) \nabla p, \nabla \theta_\ell \rangle = \lim_{\ell \to \infty}
\langle g, (\theta_\ell)^2 \rangle = \langle g, (\theta^-)^2 \rangle.
\end{equation}
 Inserting \cref{eq:A10} into \cref{eq:infbdtest2} gives 
\begin{equation}\label{eq:boundeq}
 \frac{1}{2} \frac{\mathrm{d} }{\mathrm{d} t} \| \theta^{-} \|_{2}^2  + \langle f + q \theta^{-} + \frac{1}{2} g \theta^{-}, \theta^{-} \rangle \le 0.
\end{equation}
By assumption \cref{eq:infbdassumptions} we observe that $\langle f + q \theta^{-} +\frac{1}{2} g \theta^{-} , \theta^{-} \rangle \ge 0$.
By applying this fact to \cref{eq:boundeq}, we integrate the resulting inequality over $(0,s) \subseteq [0,T]$ to get
$\| \theta^{-}(\cdot, s) \|_{2}^2 - \| \theta^{-}(\cdot,0) \|_{2}^2 \le 0$.
By assumption $\theta_0(\boldsymbol{x})>0$ for almost every $\boldsymbol{x} \in \Omega$ so one deduces that $\theta^-(\boldsymbol{x},s) = 0$ for almost every $(\boldsymbol{x},s) \in \Omega \times [0,T]$. Consequently, $0 \le \theta(\boldsymbol{x},t)$ for almost every $(\boldsymbol{x},t) \in \Omega \times [0,T]$.

To establish the upper bound, notice that by setting $\theta = (\theta - 1) + 1 = \vartheta + 1$, it is obvious that $\partial_t \theta = \partial_t \vartheta$ and 
\begin{equation}
A(1, p, \psi ; \theta) = \langle \kappa(\theta) \nabla p, \nabla \psi \rangle +
\langle q, \psi \rangle = \langle g+q,\psi \rangle,
\end{equation}
so that
\begin{equation} \label{eq:A13}
A(\theta,p,\psi ; \theta) = A(\vartheta, p, \psi ; \theta) + A(1, p, \psi ; \theta)  = A(\vartheta, p, \psi ; \theta) + \langle g+q,\psi \rangle.
\end{equation}
 Thus, in a similar fashion to what is done before, we use $\psi = \vartheta^+ = \text{max} \{ \vartheta, 0 \} \in H^1_\text{D}$ in \cref{eq:subnuma} and use \cref{eq:A13} to get
\begin{equation}
 \frac{1}{2} \frac{\mathrm{d} }{\mathrm{d} t} \| \vartheta^+ \|_{2}^2 +A(\vartheta^+, p,\vartheta^+; \theta) + \langle g+q,\vartheta^+ \rangle = \langle f, \vartheta^+ \rangle,
\end{equation} 
and thus
\begin{equation} \label{eq:A15}
\frac{1}{2} \frac{\mathrm{d} }{\mathrm{d} t} \| \vartheta^+ \|_{2}^2   + \langle
q \vartheta^{+} + \frac{1}{2} g \vartheta^{+}, \vartheta^{+} \rangle +
 \langle g + q  - f , \vartheta^+ \rangle \le 0.
\end{equation} 
By \cref{eq:infbdassumptions},
$\langle
q \vartheta^{+} + \frac{1}{2} g \vartheta^{+}, \vartheta^{+} \rangle +
 \langle g + q  - f , \vartheta^+ \rangle  \ge 0$, which after integration of \cref{eq:A15} over $(0,s) \subseteq [0,T]$ gives $\| \vartheta^+(\cdot, s) \|_{2}^2 - \| \vartheta^+(\cdot,0) \|_{2}^2 \le 0$.
By assumption $\theta_0(\boldsymbol{x})<1$ for almost every $\boldsymbol{x} \in \Omega$ so one deduces that $\vartheta^+(\dot,s) = 0$ for all $s \in [0,T]$. Consequently, $\theta(\boldsymbol{x},t) \le 1$ for almost every $(\boldsymbol{x},t) \in \Omega \times [0,T]$.

\end{document}